\documentclass[a4paper,twoside]{amsart}

\usepackage{amssymb}
\usepackage{esint}                % nach den AMS-packages laden!
\usepackage{bbm}

\newif\ifpdf
\ifx\pdfoutput\undefined
\pdffalse % we are not running PDFLaTeX
\else
\pdfoutput=1 % we are running PDFLaTeX
\pdftrue
\fi

\ifpdf
\ifx\abbrloaded\relax
    \let\next= 
\else
    \let\next=\relax
\fi

\next

\let\abbrloaded=y

\let\T=\texttt

\let\add=\advance
\let\div=\divide
\let\mul=\multiply
\let\ex=\expandafter

\catcode`\"=13
\def\lqq{``\global\let"=\rqq}
\def\rqq{''\global\let"=\lqq}
\let"=\lqq

\def\<#1>{%
    \expandafter\ifx\csname<#1>\endcsname\relax
        \errmessage{abbreviation <#1> undefined!}
    \else
        \csname<#1>\endcsname
    \fi
}
\def\abbr#1#2{%
    \expandafter\def\csname<#1>\endcsname{#2}%
}

\abbr{THANH}{H\`an Th\^e\llap{\raise 0.5ex\hbox{\'{}}} Th\`anh}
\abbr{HZ}{\textit{hz}}
\abbr{KF}{\textit{kf\kern-.05em}}
\abbr{KR}{\textit{K$\varrho$}}
\abbr{EK}{\textit{{\Large$\varepsilon$\kern-.1em}k}}
\abbr{JP}{\textit{jp}}
\abbr{MF}{\MF}
\abbr{CMR}{CMR}
\abbr{CMSS}{CMSS}
\abbr{MM}{Multiple Master}
\abbr{T1}{Type\nobreak\,1}
\abbr{PFB}{PFB}
\abbr{CM}{Computer Modern}
\abbr{TEX}{\TeX}
\abbr{LATEX}{\LaTeX}
\abbr{PDFTEX}{pdf\TeX}
\abbr{URW}{URW}
\abbr{ASCII}{ASCII}
\abbr{DTP}{DTP}
\abbr{DVI}{DVI}
\abbr{PS}{PS}
\abbr{TFM}{TFM}
\abbr{PDF}{PDF}
\abbr{HJ}{H\kern.1em\&\kern.1emJ}
\abbr{UNIX}{UNIX}
\abbr{AMIGA}{Amiga}
\abbr{DOS}{DOS}
\abbr{MAC}{Macintosh}
\abbr{C}{C}
\abbr{TRUETYPE}{TrueType}
\abbr{WIN32}{Win32}
\abbr{FPTEX}{fp\TeX}
\abbr{TETEX}{te\TeX}
\abbr{MIKTEX}{Mik\TeX}
\abbr{CMACTEX}{CMac\TeX}
\abbr{AFM}{AFM}
\abbr{MMTOOLS}{MMTOOLS}
\abbr{VNR}{VNR}
\abbr{CS}{CS}
\abbr{HTML}{HTML}
\abbr{WWW}{WWW}
\abbr{ID}{InDesign}
\abbr{ADOBE}{Adobe}
\abbr{3B2}{3B2}
\abbr{...}{\dots}
\abbr{OMEGA}{$\Omega$}
\abbr{ETEX}{$\varepsilon$-\TeX}
\abbr{NTS}{NTS}
\abbr{LF}{\textrm{\it letter\!\_\kern.1emfit}}
\abbr{SGML}{SGML}
\abbr{XML}{XML}
\abbr{NTG}{NTG}
\abbr{DANTE}{DANTE}
\abbr{GUST}{GUST}
\abbr{GUT}{GUTenberg}
\abbr{TUG}{TUG}
\abbr{BS}{\char92}
\abbr{TFTOPL}{TFtoPL}
\abbr{PDFETEX}{pdf\<ETEX>}
\abbr{ZLIB}{ZLIB}
\abbr{LIBPNG}{LIBPNG}
\abbr{LIBTIFF}{LIBTIFF}
\abbr{XPDF}{XPDF}
\abbr{MIRKA}{Miroslava Mis\'akov\'a}
\abbr{PERL}{Perl}
\abbr{MMINSTANCE}{MMInstance}
\abbr{PERL}{Perl}
\abbr{LINUX}{Linux}
\abbr{PASCAL}{Pascal}

\ifx\newcommand\undefined
    \let\next=\relax
\else
    \let\next= 
\fi
\next

\def\makeatletter{\catcode`\@11\relax}
\def\makeatother{\catcode`\@12\relax}
\makeatletter
\newcount\@tempcnta
\newcount\@tempcntb
\newif\if@tempswa
\newdimen\@tempdima
\newdimen\@tempdimb
\newdimen\@tempdimc
\newbox\@tempboxa
\newskip\@tempskipa
\newskip\@tempskipb
\newtoks\@temptokena
\makeatother

\def\setprotcode#1{%
    \rpcode#1`\!=200
    \rpcode#1`\,=330
    \rpcode#1`\-=330
    \rpcode#1`\.=330
    \rpcode#1`\;=500
    \rpcode#1`\:=500
    \rpcode#1`\?=200
    \lpcode#1`\`=700
    \rpcode#1`\'=700
    \lpcode#1 92=500  % ``
    \rpcode#1 34=500  % ''
    \rpcode#1 123=300 % --
    \rpcode#1 124=200 % ---
    \rpcode#1`\)=50
    \rpcode#1`\A=50
    \rpcode#1`\F=50
    \rpcode#1`\K=50
    \rpcode#1`\L=50
    \rpcode#1`\T=50
    \rpcode#1`\V=50
    \rpcode#1`\W=50
    \rpcode#1`\X=50
    \rpcode#1`\Y=50
    \rpcode#1`\k=50
    \rpcode#1`\r=50
    \rpcode#1`\t=50
    \rpcode#1`\v=50
    \rpcode#1`\w=50
    \rpcode#1`\x=50
    \rpcode#1`\y=50
    \lpcode#1`\(=50
    \lpcode#1`\A=50
    \lpcode#1`\J=50
    \lpcode#1`\T=50
    \lpcode#1`\V=50
    \lpcode#1`\W=50
    \lpcode#1`\X=50
    \lpcode#1`\Y=50
    \lpcode#1`\v=50
    \lpcode#1`\w=50
    \lpcode#1`\x=50
    \lpcode#1`\y=50
    \adjustprotcode#1
}

\makeatletter
\newif\ifneedadjustprotcode
\def\adjustprotcode#1{%
    \needadjustprotcodefalse
    \ifnum\pdftexversion > 13
        \ifnum \expandafter`\pdftexrevision > `g
            \needadjustprotcodetrue
        \fi 
    \else\ifnum\pdftexversion > 14
        \needadjustprotcodetrue
    \fi\fi
    \ifneedadjustprotcode
        \@tempcnta=0
        \loop
            \ifcase\lpcode#1\@tempcnta\else
                \adjustcp\lpcode#1\@tempcnta
            \fi
            \ifcase\rpcode#1\@tempcnta\else
                \adjustcp\rpcode#1\@tempcnta
            \fi
            \advance\@tempcnta 1
        \ifnum\@tempcnta < 256 \repeat
    \fi
}
\def\adjustcp#1#2#3{%
    \setbox0=\hbox{%
        \ifx#2\font\else#2\fi
        \char#3}%
    \@tempcntb=\wd0
    \mul\@tempcntb #1#2#3%
    \div\@tempcntb \fontdimen6 #2%
    #1#2#3=\@tempcntb
}
\makeatother

\fi

\swapnumbers
\theoremstyle{plain}
\newtheorem{theorem}{Theorem}[section]
\newtheorem{lemma}[theorem]{Lemma}

\theoremstyle{remark}
\newtheorem{remark}[theorem]{Remark}
\newtheorem{example}[theorem]{Example}

\theoremstyle{definition}
\newtheorem{definition}[theorem]{Definition}
\newtheorem{miniremark}[theorem]{}

\swapnumbers
\newtheoremstyle{citing}% name
  {3pt}%      Space above, empty = `usual value'
  {3pt}%      Space below
  {\itshape}% Body font
  {}%         Indent amount (empty = no indent, \parindent = para indent)
  {\bfseries}% Thm head font
  {.}%        Punctuation after thm head
  {.5em}%     Space after thm head: " " = normal interword space;
  {\thmnote{#3}}% Thm head spec

\theoremstyle{citing}
\newtheorem*{citing}{}

\theoremstyle{definition}
\newtheorem{question}{Question}

\newcounter{acerbi}
\newcounter{fusco}

\newcommand{\class}[1]{\mathcal{C}^{#1}}
\newcommand{\classification}[3]{\left \{ {#2} \in {#1} \with {#3} \right \}}

\newcommand{\Lploc}[1]{L_{\mathrm{loc}}^{#1}}

\newcommand{\oball}[3]{B^{#1}_{#3} ( #2 )}
\newcommand{\cball}[3]{\bar{B}^{#1}_{#3} ( #2 )}
\newcommand{\nat}{\mathbb{N}}
\newcommand{\integers}{\mathbb{Z}}
\newcommand{\rel}{\mathbb{R}}
\newcommand{\grass}[2]{G(#1,#2)}

\newcommand{\project}[1]{#1}
\newcommand{\eqproject}[1]{#1}
\newcommand{\pluslim}[1]{\downarrow {#1}}

\newcommand{\density}{\theta}
\newcommand{\card}{\#}
\newcommand{\unitmeasure}[1]{\omega_{#1}}
\newcommand{\besicovitch}[1]{N(#1)}
\newcommand{\isoperimetric}[1]{\gamma_{#1}}

\newcommand{\Clos}[1]{\overline{#1}}

\newcommand{\union}[2]{{\textstyle\bigcup_{#2}} #1}

\newcommand{\intersection}[2]{{\textstyle\bigcap_{#2}} #1}

\newcommand{\alt}[2]{#2}

\newcommand{\measureball}[2]{{#1}({#2})}
\newcommand{\Lp}[1]{L^{#1}}
\newcommand{\lIm}{(}
\newcommand{\rIm}{)}

\newcommand{\cube}[3]{Q_{#3}^{#1} (#2)}

\newcommand{\vdim}{n}
\newcommand{\codim}{m}
\newcommand{\adim}{{n+m}}

\newcommand{\curv}{\psi}
\newcommand{\hoelder}{\alpha}

\newcommand{\ud}{\ensuremath{\,\mathrm{d}}}
\newcommand{\printRoman}[1]{\setcounter{acerbi}{#1}\Roman{acerbi}}

\DeclareMathOperator{\without}{\sim}

\DeclareMathOperator{\tiltex}{tiltex}
\DeclareMathOperator{\heightex}{heightex}
\DeclareMathOperator{\ap}{ap}

\DeclareMathOperator{\spt}{spt}
\DeclareMathOperator{\graph}{graph}

\DeclareMathOperator{\with}{:}

\DeclareMathOperator{\dmn}{dmn}
\DeclareMathOperator{\dist}{dist}
\DeclareMathOperator{\Hom}{Hom}
\DeclareMathOperator{\Div}{div}
\newcommand{\restrict}{\mathop{\llcorner}}

\begin{document}

\ifpdf
\setprotcode \font
{\it \setprotcode \font}
{\bf \setprotcode \font}
{\bf \it \setprotcode \font}
\pdfprotrudechars=1
\fi

\author{Ulrich Menne}
\thanks{The author acknowledges financial support via
the DFG Forschergruppe 469.}
\thanks{\textit{AEI publication number.} AEI-2008-063.}
\address{Max-Planck-Institut f\"ur Gravitationsphysik (Albert-Einstein-Institut) \\ Am M{\"u}h\-len\-berg 1 \\ D-14476 Golm \\ Deutschland}
\email{Ulrich.Menne@aei.mpg.de}
\title[Some Applications of the isoperimetric Inequality]{Some Applications of the isoperimetric Inequality for integral Varifolds}
\date{\today}
\begin{abstract}
	In this work the Isoperimetric Inequality for integral varifolds is
	used to obtain sharp estimates for the size of the set where the
	density quotient is small and to generalise Calder\'on's and Zygmund's
	theory of first order differentiability for functions in Lebesgue
	spaces from Lebesgue measure to integral varifolds.
\end{abstract}
\subjclass[2000]{Primary 49Q15; Secondary 26B35}

\maketitle

\tableofcontents

\section*{Introduction}
In this work, and also in the work in \cite{snulmenn:poincare} and
\cite{snulmenn:c2} depending on it, weak notions of regularity for integral
varifolds in an open subset of Euclidian space whose distributional first
variation is given by either a Radon measure or a locally to the $p$-th power
summable function, $1 < p \leq \infty$, are investigated. As it is well
known, see e.g. \cite[8.1\,(2)]{MR0307015}, even in the second case with
$p=\infty$ the singular set where the support does not locally correspond to a
submanifold of class $\class{1}$ may have positive measure. Therefore the
notions of regularity studied here are decay rates of height-excess and
tilt-excess which provide a way to quantify the amount of flatness entailed by
the conditions on the mean curvature near almost every point.

Next, in order to precisely state the problem and the results, some
definitions will be recalled. Suppose throughout the introduction that $\vdim,
\codim \in \nat$, $U$ is an open subset of $\rel^\adim$, and $\mu$ is an
integral $\vdim$ varifold in $U$, i.e., using \cite[Theorem 11.8]{MR87a:49001}
as a definition, $\mu$ is a Radon measure on $U$ and for $\mu$ almost all $x
\in U$ there exists an approximate tangent plane $T_x \mu \in \grass{\adim
}{\vdim}$ with multiplicity $\density^\vdim ( \mu, x ) \in \nat$ of $\mu$ at
$x$, $\grass{\adim}{\vdim}$ denoting the set of $\vdim$ dimensional,
unoriented planes in $\rel^\adim$. The distributional first variation of mass
of $\mu$ equals
\begin{gather*}
        ( \delta \mu ) ( \eta ) = {\textstyle\int} \Div_\mu \eta \ud \mu \quad
        \text{whenever $\eta \in \alt{\mathcal{D}}{C_\mathrm{c}^1} ( U,
        \rel^\adim )$}
\end{gather*}
where $\Div_\mu \eta (x)$ is the trace of $D\eta(x)$ with respect to $T_x
\mu$. $\| \delta \mu \|$ denotes the total variation measure associated to
$\delta \mu$ and $\mu$ is said to be of locally bounded first variation if and
only if $\| \delta \mu \|$ is a Radon measure, in this case the generalised
mean curvature vector $\vec{\mathbf{H}}_\mu (x) \in \rel^\adim$ can be defined
by the requirement
\begin{gather*}
	\vec{\mathbf{H}}_\mu (x) \bullet v = - \lim_{\varrho \pluslim{0}}
	\frac{( \delta \mu ) ( \chi_{\oball{}{x}{\varrho}} v
	)}{\measureball{\mu}{\oball{}{x}{\varrho}}} \quad \text{for $v \in
	\rel^\adim$}
\end{gather*}
whenever these limits exist for $x \in U$; here $\bullet$ denotes the usual
inner product on $\rel^\adim$. Moreover, $\mu$ is said to satisfy
\eqref{eqn_hp}, $1 \leq p \leq \infty$, if and only if it is of locally
bounded first variation, $\vec{\mathbf{H}}_\mu \in \Lploc{p} ( \mu,
\rel^\adim)$, and, in case $p > 1$, satisfies
\begin{gather} \label{eqn_hp}
	( \delta \mu ) ( \eta ) = - {\textstyle\int} \vec{\mathbf{H}}_\mu
	\bullet \eta \ud \mu \quad \text{whenever $\eta \in
	\alt{\mathcal{D}}{C_\mathrm{c}^1} ( U, \rel^\adim  )$} \tag{$H_p$}.
\end{gather}
Also, adapting Anzellotti's and Serapioni's definition in \cite{MR1285779},
$\mu$ is called countably rectifiable of class $\mathcal{C}^2$, or for short
$\mathcal{C}^2$ rectifiable, if and only if $\mu$ almost all of $U$ can be
covered by a countable collection of $\vdim$ dimensional submanifolds of class
$\class{2}$. The notation follows \cite{MR87a:49001} which includes a list of
basic notation on page (vii).

The following questions arise.
\begin{question}
	Suppose $\vdim, \codim \in \nat$, $1 \leq p \leq \infty$, $0 < \alpha
	\leq 1$, and $1 \leq q \leq \infty$. Does the condition \eqref{eqn_hp}
	on an integral $\vdim$ varifold $\mu$ in $U$, $U$ a nonempty, open
	subset of $\rel^\adim$, imply
	\begin{gather*}
		\limsup_{\varrho \pluslim 0} \varrho^{-1-\alpha-\vdim/q}
		\| \dist ( \cdot - x, T_x \mu ) \|_{\Lp{q} ( \mu \restrict
		\oball{}{x}{\varrho})} < \infty \quad \text{for $\mu$ almost
		all $x \in U$?}
	\end{gather*}
\end{question}
\begin{question}
	Suppose $\vdim, \codim \in \nat$, $1 \leq p \leq \infty$, $0 < \alpha
	\leq 1$, and $1 \leq q \leq \infty$. Does the condition \eqref{eqn_hp}
	on an integral $\vdim$ varifold $\mu$ in $U$, $U$ a nonempty, open
	subset of $\rel^\adim$, imply
	\begin{gather*}
		\limsup_{\varrho \pluslim 0} \varrho^{-\alpha-\vdim/q}
		\| T_\cdot \mu - T_x \mu \|_{\Lp{q} ( \mu \restrict
		\oball{}{x}{\varrho})} < \infty \quad \text{for $\mu$ almost
		all $x \in U$?}
	\end{gather*}
	Here $S \in \grass{\adim}{\vdim}$ is identified with the element of
	$\Hom ( \rel^\adim, \rel^\adim )$ given by the orthogonal projection
	of $\rel^\adim$ onto $S$.
\end{question}
Clearly, the two questions are related by Caccioppoli type inequalities, see
e.g. in \cite[5.5]{MR485012}, at least in the case $q=2$ where the quantities
considered agree with the classical tilt and height excess. Also note that an
affirmative answer to one of the questions with $\alpha$, $q$ implies an
affirmative answer to the same question for any $0 < \alpha' \leq 1$, $q < q'
< \infty$ such that $\alpha q = \alpha ' q '$ by use of the trivial
$\Lp{\infty}$ bounds of the functions involved. The case $\alpha = 1$ is of
particular interest in both questions. A varifold satisfying the decay
estimate in the first question with $\alpha=1$ and $q=1$ is $\mathcal{C}^2$
rectifiable, see \cite[Appendix A]{rsch:willmore}. In the second question the
case $\alpha = 1$ is related to the local computability of the mean curvature
vector from the geometry of $\classification{U}{x}{\density^\vdim ( \mu, x )
\geq 1}$, see \cite[Lemma 6.3]{MR1906780} (or \cite[Prop.  6.1]{MR2064971} or
\cite[Theorem 4.1]{rsch:willmore}). On the other hand the quantity $\alpha q$
to some extend determines how well $\mu$ can be approximated by multivalued
graphs near generic points, see Almgren \cite[Chapter 3]{MR1777737} and Brakke
\cite[Chapter 5]{MR485012} and also the forthcoming paper
\cite{snulmenn:poincare}. Such kind of approximation has been fundamental for
regularity investigations, for example, in the work of Almgren in
\cite{MR1777737}.

Next, an overview of results concerning these two questions will be given.
Brakke answers both questions in the affirmative for any $\vdim, \codim \in
\nat$ if
\begin{gather*}
	\text{either} \quad \text{$p=1$, $\alpha=1/2$, $q=2$} \qquad \text{or}
	\quad \text{$p=2$, $\alpha < 1$, $q=2$}
\end{gather*}
in \cite[5.7]{MR485012}. Sch{\"a}tzle provides a positive answer in the case
\begin{gather*}
	\text{$m=1$, $p > \vdim$, $p \geq 2$, $\alpha = 1$, $q = \infty$}
\end{gather*}
for the first question and in the case
\begin{gather*}
	\text{$m=1$, $p > \vdim$, $p \geq 2$, $\alpha = 1$, $q = 2$}
\end{gather*}
for the second question, see \cite[Prop. 4.1, Thm. 5.1]{MR2064971}. Moreover,
in subsequent work Sch{\"a}tzle showed for arbitrary dimensions that the decay
rates occuring in the two questions hold if
\begin{gather*}
	\text{$p=2$, $\alpha=1$, $q=2$}
\end{gather*}
provided $\mu$ is additionally assumed to be $\mathcal{C}^2$ rectifiable.

In this paper, it is shown by an example of a unit density, $\mathcal{C}^2$
rectifiable $\vdim$ varifold in $\rel^{\vdim+1}$ that the answers to both
questions is in the negative if $p < \vdim$ and $\alpha q > \frac{\vdim
p}{\vdim - p}$, see \ref{example:scaling}. In particular, in case $1 \leq p <
\frac{2 \vdim}{\vdim + 2}$ proving appropriate decay for the classical
height-excess or tilt-excess, i.e. answering the first or second question in
the affirmative for $\alpha = 1$, $q = 2$, cannot serve as
an intermediate step in studying $\mathcal{C}^2$ rectifiability or local
computability of the mean curvature vector. This was the original motivation
to consider exponents $q \neq 2$.

In order to provide new cases where the questions are answered in the
affirmative, it will turn out to be useful in
\cite{snulmenn:poincare,snulmenn:c2} to have a theory of first order
differentiability for functions in $\Lp{p} ( \mu )$, $\mu$ an integral $\vdim$
varifold, similar to the one developed by Calder\'on and Zygmund in
\cite{MR0136849} for $\Lp{p} ( \mathcal{L}^\vdim )$, at one's disposal.
However, other kinds of applications may occur in the future. The key to carry
over this theory from the Lebesgue measure case to the case of integral
varifolds is the following differentiation theorem which corresponds to
\cite[Theorem 10\,(ii)]{MR0136849} but whose proof uses techniques employed by
Federer in \cite[2.9.17]{MR41:1976}.
\begin{citing} [Theorem \ref{thm:big_o_little_o}]
	Suppose $\vdim, \codim \in \nat$, $1 \leq p \leq \vdim$, $U$ is an
	open subset of $\rel^\adim$, $\mu$ is an integral $\vdim$ varifold in
	$U$ satisfying \eqref{eqn_hp}, $\nu$ measures $U$ with $\nu ( U
	\without \spt \mu ) = 0$, $A$ is $\mu$ measurable with $\nu (A) = 0$,
	and $1 \leq q < \infty$. In case $p < \vdim $ additionally suppose for
	some $1 \leq r \leq \infty$ and some nonnegative function $f \in
	\Lp{r}_\mathrm{loc} ( \mu)$ that
	\begin{gather*}
		\nu = f \mu \quad \text{and} \quad q \leq 1 + ( 1 - 1/r )
		\frac{p}{\vdim -p}.
	\end{gather*}

	Then for $\mathcal{H}^\vdim $ almost all $a \in A$
	\begin{gather*}
		\limsup_{s \pluslim{0}} \nu ( \cball{}{a}{s} ) \big / s^{\vdim
		q} \quad \text{equals either $0$ or $\infty$}.
	\end{gather*}
\end{citing}
The bound on $q$ is sharp as demonstrated in \ref{remark:precise1},
\ref{remark:precise2}. Its occurance is due to the fact that in case $p <
\vdim$ the number $\vdim^2/(\vdim-p)$ in the following proposition cannot be
replaced by any larger number, see \ref{example:scaling}: \emph{Suppose
$\vdim, \codim \in \nat$, $1 \leq p < \vdim$, $\mu$ is an integral $\vdim$
varifold in $\rel^\adim$ satisfying \eqref{eqn_hp}, then for $\mu$ almost all
$a \in U$ there exists $\varepsilon
> 0$ such that $$\lim_{r \pluslim{0}} \frac{\mu \big ( \cball{}{a}{r} \without
\{ x \with \text{$\measureball{\mu}{\cball{}{x}{\varrho}} \geq c_\vdim
\varrho^\vdim$ for $0 < \varrho < \varepsilon$} \}
\big)}{r^{\vdim^2/(\vdim-p)}} = 0$$ where $c_\vdim$ is a positive, finite
number depending only on $\vdim$}, see \ref{thm:structure_lower_density},
\ref{remark:some_explanations}. Similar propositions with $\vdim^2/(\vdim-p)$
replaced by any slighly smaller number can be obtained via integration of the
monotonicity formula, see \cite[Theorem 17.6]{MR87a:49001}.  The optimal bound
is derived using the Isoperimetric Inequality. All these results will be
proved under the weaker condition $\density^\vdim ( \mu, x ) \geq 1$ for $\mu$
almost every $x \in U$ replacing the integrality condition on $\mu$.

The work is organised as follows. In the first section the example is
constructed. In the second section the Isoperimetric Inequality is used to
derive some sharp bounds on the size of the set where the $\vdim$ density
ratio is small and in the last section a theory of first order differentiation
in Lebesgue spaces defined with respect to a varifold is presented. Note that
with the exception of \ref{def:fint}--\ref{remark:final} the work was part of
the author's PhD thesis, see \cite{mydiss}.

For notation of geometric measure theory see \cite{MR87a:49001,MR41:1976}.

\textbf{Acknowledgements.} The author offers his thanks to Professor Reiner
Sch\"atz\-le under whose guidance the underlying PhD thesis was written. The
author also would like to thank Professor Tom Ilmanen for inviting him to the
ETH Z\"urich in 2006.

\section{An example concerning height and tilt decays of integral varifolds} \label{app:example}
In this section a family of integral $\vdim$ varifolds with prescribed decay
rates of height and tilt quantities is constructed. In fact, the decay rate
for tilt can be arranged to be slightly larger than the one of the height with
the same exponent. However, this feature that will only become relevant in
\cite{snulmenn:poincare}.
\begin{definition}
	Suppose $x \in \rel^\adim $ and $0 < \varrho < \infty$.

	Then $\cube{}{x}{\varrho} := \{ y \in \rel^\adim  \with \text{$| y_i -
	x_i | \leq \varrho$ for $i = 1, \ldots, \adim $}
	\}$\index{$\cube{}{x}{\varrho}$}. To avoid ambiguity, $\cube{\adim
	}{0}{\varrho}$ will be written instead of $\cube{}{0}{\varrho}$.
\end{definition}
\begin{example} \label{example:scaling}
	Suppose $\vdim  \in \nat$, $1 \leq p < \vdim $, $0 < \hoelder_i \leq 1$,
	$1 \leq q_i < \infty$ for $i \in \{ 1, 2 \}$, such that
	\begin{gather*}
		\hoelder_2 q_2 \leq \hoelder_1 q_1, \quad \frac{1}{p} > 1 +
		\frac{\hoelder_2 q_2}{\hoelder_1 q_1} \Big ( \frac{1}{n} +
		\frac{1}{\hoelder_2 q_2} - 1 \Big )
	\end{gather*}
	In case $\hoelder_1 q_1 = \hoelder_2 q_2$ the last condition reads $\hoelder_2
	q_2 > \frac{\vdim p}{\vdim-p}$.

	Then there exists a rectifiable $\vdim $ varifold $\mu$ in
	$\rel^{\vdim +1}$, $T \in G(\vdim +1,\vdim )$ and $0 < \Gamma <
	\infty$ with the following properties:
	\begin{enumerate}
		\item \label{item:scaling:analytic} $T \subset \spt \mu$
		and $(\spt \mu) \without T$ is an $\vdim $ dimensional
		manifold of class \alt{$\infty$}{$\mathcal{C}^\infty$}.
		\item \label{example:scaling:basic} $\density^\vdim  ( \mu, x
		) = 1$ for $x \in \spt \mu$ and $T_x \mu = T$ for $x \in T$.
		\item \label{example:scaling:delta_mu} For some
		$\vec{\mathbf{H}}_\mu \in \Lp{p}_\mathrm{loc} ( \mu,
		\rel^{\vdim +1} )$ there holds $(\delta\mu) ( \eta ) = - \int
		\vec{\mathbf{H}}_\mu \bullet \eta \ud \mu$ whenever $\eta \in
		\alt{\mathcal{D}}{C_\mathrm{c}^1} ( \rel^{\vdim +1},
		\rel^{\vdim +1})$.
		\item \label{item:scaling:estimates} Whenever $x \in T$ and $0
		< \varrho \leq 1$
		\begin{gather*}
			\Gamma^{-1} \varrho^{\hoelder_2 q_2} \leq \varrho^{-\vdim }
			\mu ( \{ \xi \in \cball{}{x}{\varrho} \with \dist (
			\xi - x , T ) \geq \varrho/\Gamma \} ), \\
			\varrho^{-\vdim } \mu ( \cball{}{x}{\varrho} \without
			T ) \leq \Gamma \varrho^{\hoelder_2 q_2}, \\
			\varrho^{-1-\vdim /q_2} \big (
			{\textstyle\int_{\cball{}{x}{\varrho}}} \dist ( \xi -
			x, T_x \mu )^{q_2} \ud \mu ( \xi ) \big)^{1/q_2} \approx
			\varrho^{\hoelder_2}, \\
			\varrho^{-\vdim /q_1} \big (
			{\textstyle\int_{\cball{}{x}{\varrho}}} |
			\eqproject{T_\xi\mu} - \eqproject{T_x \mu} |^{q_1} \ud
			\mu ( \xi ) \big )^{1/q_1} \approx
			\varrho^{\hoelder_1},
		\end{gather*}
		here $a \approx b$ means that $a \leq \Gamma_1 b$ and $b \leq
		\Gamma_1 a$ for some positive, finite number $\Gamma_1$
		depending only on $\vdim $, and $\hoelder_i$, $q_i$ for $i \in \{
		1,2 \}$.
		\item \label{example:scaling:last} Whenever $1 < r < \infty$,
		$\vdim  + ( 1 - 1/r ) \hoelder_2 q_2 < s < \infty$ there
		exists a nonnegative function $f \in \Lp{r}_\mathrm{loc} (
		\mu)$ such that $f(x) = 0$ for $x \in T$, and
		\begin{gather*}
			\varrho^s \approx
			{\textstyle\int_{\cball{}{x}{\varrho}}} f \ud \mu
			\quad \text{whenever $x \in T$, $0 < \varrho \leq 1$},
		\end{gather*}
		here $a \approx b$ means $a \leq \Gamma_2b$ and $b \leq
		\Gamma_2 a$ for some positive, finite number $\Gamma_2$
		depending only on $\vdim $ and $s$.
	\end{enumerate}
\end{example}
\begin{proof} [Construction of example]
	Let $a := \hoelder_2 q_2 / n + 1$, $b := ( \hoelder_1q_1- \hoelder_2 q_2 ) / a
	+ 1 \geq 1$. Define for $i \in \nat_0$
	\begin{gather*}
		W_i := \big \{ \cube{}{x}{2^{-i-2}} \with 2^{i+1} x \in
		\integers^\vdim  \big \}.
	\end{gather*}
	Clearly, $\union{\Clos{Q}}{Q \in W_i} = \rel^\vdim $ and $W_i$ is
	\alt{disjointed}{pairwise disjoint}. Let
	\begin{gather*}
		F_i := \big \{ ] 2^{-i-1}, 2^{-i} [ \times W \with W \in W_i
		\big \} \quad \text{for $i \in \nat_0$}, \quad F :=
		\union{F_i}{i \in \nat_0}.
	\end{gather*}
	Clearly, $\union{\Clos{S}}{S \in F} = ]0,1] \times \rel^\vdim $ and
	$F$ is \alt{disjointed}{pairwise disjoint}. Let $T := \{ 0 \} \times
	\rel^\vdim $.

	Next, it will be indicated how to construct for every $0 < \sigma
	\leq \varrho < \infty$ a compact $\vdim $ dimensional submanifold $M$ of
	$\rel^{\vdim +1}$ of class \alt{$\infty$}{$\mathcal{C}^\infty$} such
	that
	\begin{gather*}
		M \subset \cube{\vdim +1}{0}{\varrho}, \quad (\Gamma_0)^{-1} \varrho^\vdim
		\leq \mathcal{H}^\vdim  (M) \leq \Gamma_0 \varrho^\vdim , \quad |
		\vec{\mathbf{H}}_M | \leq \Gamma_0 \sigma^{-1}, \\
		\mathcal{H}^\vdim  ( \{ x \in M \with | T_x M - T | \geq 1 \}
		) \geq (\Gamma_0)^{-1} \sigma \varrho^{\vdim -1}, \\
		\mathcal{H}^\vdim  ( \{ x \in M \with
		\text{$\vec{\mathbf{H}}_M (x) \neq 0$ or $T_x M \neq T$} \} )
		\leq \Gamma_0 \sigma \varrho^{\vdim -1}
	\end{gather*}
	where $\Gamma_0$ is a positive, finite number depending only on $\vdim
	$.  To construct $M$, one may assume $\varrho=1$. Choose a concave function
	$f : [-1/2,1/2] \to [0,1]$ and $0 < \Gamma_1 < \infty$ such that
	\begin{gather*}
		f ( -1/2 ) = \sigma/4 = f (1/2), \\
		f (s) = \sigma/2 \quad \text{whenever $s \in
		[-1/2+\sigma/4,1/2-\sigma/4]$}
	\end{gather*}
	and such that
	\begin{gather*}
		N := \{ (s,t) \in [-1/2,1/2] \times \rel \with |t|=f(s) \}
		\cup (\{-1/2,1/2\} \times [-\sigma/4,\sigma/4])
	\end{gather*}
	is a $1$ dimensional submanifold of class
	\alt{$\infty$}{$\mathcal{C}^\infty$} with $| \vec{\mathbf{H}}_N | \leq
	\Gamma_1 \sigma^{-1}$. Noting
	\begin{gather*}
		\mathcal{H}^1 ( \alt{}{\graph} f | [-1/2,-1/2+\sigma/4] \cap
		[ 1/2-\sigma/4, 1/2] ) \leq \sigma,
	\end{gather*}
	one can take
	\begin{gather*}
		M := \{ ( y,z ) \in \rel \times \rel^\vdim  \with ( |z|, y )
		\in N \}.
	\end{gather*}

	For each $i \in \nat_0$ and $Q \in F_i$ choose a $\vdim $ dimensional
	submanifold $M_Q$ of the type just constructed corresponding to
	$\varrho_i := 2^{-i a-2}$, $\sigma_i := 2^{-i b a -2}$ contained in
	$Q$ and let $M$ be the union those submanifolds. Take $\mu :=
	\mathcal{H}^\vdim  \restrict ( T \cup M )$.
	\eqref{item:scaling:analytic} is now evident.

	To prove the estimates, fix $x \in T$ and define for $i,j \in \nat_0$
	\begin{gather*}
		b_{i,j} := \card \big \{ Q \in F_j \with Q \cap
		\cube{}{x}{2^{-i}} \neq \emptyset \big \}, \alt{\\}{\quad}
		c_{i,j} := \card \big \{ Q \in F_j \with Q \subset
		\cube{}{x}{2^{-i}} \big \}.
	\end{gather*}
	Clearly, $b_{i,j} = c_{i,j} = 0$ if $j < i$. If $j \geq i$, one
	estimates
	\begin{gather*}
		b_{i,j} \leq \big ( 2^{j-i+2} + 1 \big )^\vdim  \leq \big ( 5
		\cdot 2^{j-i} \big )^\vdim , \quad c_{i,j} \geq \big (
		2^{j-i+2} - 1 )^\vdim \geq \big ( 3 \cdot 2^{j-i} \big )^\vdim
		.
	\end{gather*}

	One calculates
	\begin{gather*}
		\mu ( \cube{}{x}{2^{-i}} \without T ) \leq \sum_{j=0}^\infty
		b_{i,j} \Gamma_0 ( \varrho_j )^\vdim  \leq (5/4)^\vdim  \Gamma_0
		(2^{-i})^{a \vdim } ( 1 - 2^{\vdim (1-a)})^{-1}, \\
		n-ba(1-p)+(1-n)a = - \hoelder_1 q_1 + p ( \hoelder_1 q_1 - \hoelder_2
		q_2 + \hoelder_2 q_2/n+1 ) < 0, \\
		{\textstyle\int_{\cube{}{x}{2^{-i}}\without T}} |
		\vec{\mathbf{H}}_M |^p \ud \mu \leq \sum_{j=0}^\infty b_{i,j}
		( \Gamma_0 )^{p+1} ( \sigma_j )^{1-p} ( \varrho_j )^{\vdim -1}
		\\
		\leq 5^\vdim  ( \Gamma_0 )^{\vdim +1} ( 2^{-i} )^{b a (
		1 - p ) + \vdim  - 1} ( 1 - 2^{\vdim -ba(1-p)+(1-\vdim
		)a} )^{-1} < \infty, \\
		{\textstyle\int_{\cube{}{x}{2^{-i}}}} \dist ( \xi - x, T
		)^{q_2} \ud \mu ( \xi ) \leq 2^{-iq_2} \mu ( \cube{}{x}{2^{-i}} \without
		T ), \\
		{\textstyle\int_{\cube{}{x}{2^{-i}}}} | \eqproject{T_\xi \mu}
		- \project{T} |^{q_1} \ud \mu ( \xi ) \leq (2\vdim )^{q_1}
		\sum_{j=0}^\infty b_{i,j} \Gamma_0 \sigma_j (
		\varrho_j)^{\vdim -1} \\
		\leq (2\vdim )^{q_1} (5/4)^\vdim  \Gamma_0 ( 2^{-i} )^{b a+a
		(\vdim -1)} ( 1 - 2^{\vdim -ba-a(\vdim -1)})^{-1}, \\
		2^{(i+1)q_2}
		{\textstyle\int_{\cube{}{x}{2^{-i}}}} \dist ( \xi - x , T
		)^{q_2}
		\ud \mu ( \xi ) \\ \geq \mu ( \{ \xi \in \cube{}{x}{2^{-i}}
		\with \dist ( \xi - x, T ) \geq 2^{-i-1} \} ) \geq
		(\Gamma_0)^{-1} ( \varrho_i )^\vdim  = ( 4^\vdim  \Gamma_0)^{-1}
		2^{-ia \vdim }, \\
		{\textstyle\int_{\cube{}{x}{2^{-i}}}} | \eqproject{T_\xi \mu}
		- \project{T} |^{q_1} \mu ( \xi ) \geq ( \Gamma_0 )^{-1} \sigma_i
		( \varrho_i)^{\vdim -1} = ( 4^\vdim  \Gamma_0 )^{-1} (
		2^{-i})^{ab+a(\vdim -1)}.
	\end{gather*}
	Therefore \eqref{example:scaling:delta_mu} and
	\eqref{item:scaling:estimates} are proved and the first estimate of
	\eqref{item:scaling:estimates} implies
	\eqref{example:scaling:basic}.
	
	To prove \eqref{example:scaling:last}, define $f$ by $f(y) :=
	2^{(\vdim a-s)i}$ if $y \in \alt{\bigcup F_i}{\bigcup_{S \in F_i}
	S}$ for some $i \in \nat_0$ and $f(y) = 0$ else. Then for $i \in
	\nat_0$
	\begin{gather*}
		{\textstyle\int_{\cube{}{x}{2^{-i}}}} |f| \ud \mu \leq
		\sum_{j=0}^\infty b_{i,j} 2^{(\vdim a - s)j} \Gamma_0
		(\varrho_j)^\vdim \leq (5/4)^\vdim  \Gamma_0 ( 2^{-i} )^s ( 1 -
		2^{\vdim -s} )^{-1}, \\
		{\textstyle\int_{\cube{}{x}{2^{-i}}}} |f|^r \ud \mu \leq
		\sum_{j=0}^\infty b_{i,j} 2^{(\vdim  a-s)rj} \Gamma_0 ( \varrho_j
		)^\vdim \\
		\leq (5/4)^\vdim  \Gamma_0 ( 2^{-i} )^{(s-\vdim a)r+\vdim
		a} ( 1 - 2^{\vdim +(\vdim a-s)r-\vdim a} )^{-1} <
		\infty
	\end{gather*}
	because
	\begin{gather*}
		n + (na-s)r-an = \hoelder_2 q_2 ( r-1 ) + r (n-s) < 0.
	\end{gather*}
	The estimate from below is similar to the one from above.
\end{proof}
\begin{remark}
	The integral $n$ varifold $\mu$ constructed depends only on $n$ and
	the products $\hoelder_i q_i$ for $i \in \{ 1, 2 \}$. Moreover, the
	assumption $\hoelder_i \leq 1$ for $i \in \{ 1, 2 \}$ could be
	replaced by $\hoelder_i < \infty$ for $i \in \{ 1, 2 \}$.
\end{remark}
\begin{remark} \label{remark:best_rates}
	Taking $p=1$, $\hoelder_1 = \hoelder_2$, and $q_1=q_2=2$ in the last
	two estimates of \eqref{item:scaling:estimates} shows that for every
	$\vdim \in \nat$, $\vdim > 1$,
		$1/2+(2(\vdim-1))^{-1} < \hoelder \leq 1$,
	there exists an integral $n$ varifold $\mu$ of $\rel^{\vdim + 1}$ of
	locally bounded first variation such that for some $A$ with $\mu (A) >
	0$
	\begin{gather*}
		\lim_{\varrho \pluslim{0}} \varrho^{-2\hoelder} \heightex_\mu
		( x, \varrho, T_x \mu ) = \infty, \quad \lim_{\varrho
		\pluslim{0}} \varrho^{-2\hoelder} \tiltex_\mu ( x, \varrho,
		T_x \mu ) = \infty
	\end{gather*}
	for $x \in A$. In \cite[5.7]{MR485012} Brakke showed in arbitrary
	codimension that the above limits equal $0$ almost everywhere with
	respect to $\mu$ if $\hoelder = 1/2$.
\end{remark}
\begin{remark} \label{remark:no_equiv_c2_tilt}
	Similarly to the preceding remark, taking $\hoelder_1 = \hoelder_2 =
	1$, $q_1 = q_2 = q$ and noting \eqref{item:scaling:analytic},
	one obtains for every $p^\ast = \frac{np}{n-p} < q < \infty$ an
	integral $n$ varifold $\mu$ satisfying \eqref{eqn_hp} which is
	countably rectifiable of class $\mathcal{C}^2$ such that for some $A$
	with $\mu (A) > 0$
	\begin{gather*}
		\lim_{\varrho \pluslim{0}} \varrho^{-2-n/q} \big (
		{\textstyle\int_{\oball{}{x}{\varrho}}} \dist ( \xi - x , T_x
		\mu )^q \ud \mu ( \xi ) \big )^{1/q} = \infty, \\
		\lim_{\varrho \pluslim{0}} \varrho^{-1-n/q} \big (
		{\textstyle\int_{\oball{}{x}{\varrho}}} | \eqproject{T_\xi
		\mu} - \eqproject{T_x \mu} |^q \ud \mu ( \xi ) \big )^{1/q} =
		\infty
	\end{gather*}
	for $x \in A$. In particular, if $p < \frac{2 \vdim}{\vdim +2}$ then
	countable rectifiability of class $\mathcal{C}^2$ does not imply
	quadratic decay of neither $\tiltex_\mu$ nor $\heightex_\mu$. If $p =
	2$, countable rectifiability of class $\mathcal{C}^2$ is equivalent to
	quadratic decay of both quantities, see \cite[Theorem
	3.1]{rsch:willmore}.
\end{remark}
\section{The size of the set where the $\vdim$ density quotient is small} \label{app:iso}
In this section the Isoperimetric Inequality is used to derive basic facts on
the size of the set where the $\vdim$ density quotient is small. Although the
general procedure of such estimates is clearly known, see
\ref{miniremark:good_point}, it appears to be rarely used in literature. The
sharpness of the results is necessary to determine the precise limiting
exponent up to which the differentiation theory in the next section can be
developed. Similarly, the accuracy of the bounds obtained in
\cite{snulmenn:poincare} depends on the results of this section.
\begin{miniremark} \label{miniremark:situation1}
	The following situation will be studied:

	\alt{$\vdim, \adim \in \nat$, $\vdim < \adim$,}{$\codim, \vdim \in
	\nat$,} $1 \leq p \leq \vdim$,
	$U$ is an open subset of $\rel^\adim$, $\mu$ is a rectifiable $\vdim$
	varifold\footnote{Note that a definition of a rectifiable $\vdim$
	varifold results from the definition of an integral $\vdim$ varifold
	through replacement of the condition $\density^\vdim ( \mu, x ) \in
	\nat$ by $0 < \density^\vdim ( \mu, x ) < \infty$.} in $U$ of locally
	bounded first variation, $\density^\vdim ( \mu, x ) \geq 1$ for $\mu$
	almost all $x \in U$, and, in case $p > 1$,
	\begin{gather*}
		( \delta \mu ) ( \eta) = - {\textstyle\int}
		\vec{\mathbf{H}}_\mu \bullet \eta \ud \mu \quad \text{whenever
		$\eta \in \alt{\mathcal{D}(U,\rel^\adim)}{C_\mathrm{c}^1 ( U,
		\rel^\adim )}$}
	\end{gather*}
	for some $\vec{\mathbf{H}}_\mu \in \Lp{p}_\mathrm{loc} ( \mu,
	\rel^\adim)$. In doing so, the following abbreviation will be used:
	\begin{gather*}
		\curv = \| \delta \mu \| \quad \text{if $p = 1$}, \quad
		\curv = | \vec{\mathbf{H}}_\mu |^p \mu \quad \text{else}.
	\end{gather*}
\end{miniremark}
\begin{theorem} [Isoperimetric Inequality for varifolds] \label{thm:basic_sobolev}
	Suppose \alt{$m, n \in \nat$, $m \leq n$,}{$\codim \in \nat_0$, $\vdim
	\in \nat$,} $\mu$ is a rectifiable $\vdim$ varifold in $\rel^\adim $
	with $\mu ( \rel^\adim  ) < \infty$ and $\| \delta \mu \| ( \rel^\adim
	) < \infty$.  

	Then for some positive, finite number $\gamma$ depending only on
	$\vdim $
	\begin{gather*}
		\mu \big ( \big \{ x \in \rel^\adim  \with \density^\vdim  (
		\mu, x ) \geq 1 \big \} \big ) \leq \gamma \, \mu ( \rel^\adim
		)^{1/\vdim } \| \delta \mu \| ( \rel^\adim  ).
	\end{gather*}
\end{theorem}
\begin{proof}
	This follows from \cite[Theorem 7.1]{MR0307015} with a constant
	$\gamma$ depending on $\adim $ (what would be sufficient for the purpose of
	this work). A slight modification of \cite[Lemma 18.7, Theorem
	18.6]{MR87a:49001} yields the stated result.
\end{proof}
\begin{definition} \label{def:isoperimetric}
	For $\vdim  \in \nat$ let $\isoperimetric{\vdim
	}$\index{$\isoperimetric{\vdim }$} denote the best constant $\gamma$
	in \ref{thm:basic_sobolev}.
\end{definition}
\begin{remark}
	Taking \alt{$\adim=\vdim$}{$\codim=0$}, $\mu = \mathcal{L}^\vdim
	\restrict \cball{\vdim }{0}{1}$ yields
	\begin{gather*}
		\isoperimetric{\vdim } \geq \unitmeasure{\vdim }^{-1/\vdim } /
		\vdim .
	\end{gather*}
	Does equality hold?
\end{remark}
\begin{miniremark} \label{miniremark:good_point} 
	An important consequence of the Isoperimetric Inequality
	\ref{thm:basic_sobolev} and the starting point for the estimates in
	the present section is the following fact which can be derived by a
	variant of \cite[5.1.6]{MR41:1976} or \cite[8.3]{MR0307015}, see
	\cite[Prop.  3.1]{leomas08} or \cite[A.8, A.9]{mydiss}.

	\emph{Suppose $\vdim$, $\codim$, $p=1$, $U = \oball{}{a}{r}$ for some
	$a \in \rel^\adim$ and $0 < r < \infty$, and $\mu$ are as in
	\ref{miniremark:situation1}, $a \in \spt \mu$, and $$\measureball{\|
	\delta \mu \|}{\cball{}{a}{\varrho}} \leq ( 2
	\isoperimetric{\vdim})^{-1} \mu ( \cball{}{a}{\varrho})^{1-1/\vdim}
	\quad \text{whenever $0 < \varrho < r$},$$ then
	$$\measureball{\mu}{\cball{}{a}{\varrho}} \geq ( 2 \vdim
	\isoperimetric{\vdim})^{-\vdim} \varrho^\vdim \quad \text{whenever $0
	< \varrho < r$}.$$} Also note, if $p=\vdim>1$ or $p=\vdim=1$ and $\|
	\delta \mu \| ( \{ a \}) < ( 2 \isoperimetric{\vdim} )^{-\vdim}$, then
	\begin{gather*}
		\measureball{\| \delta \mu \|}{\cball{}{x}{\varrho}} \leq ( 2
		\isoperimetric{\vdim} )^{-1} \mu ( \cball{}{x}{\varrho} )
		^{1-1/\vdim}
	\end{gather*}
	whenever $0 < \varrho < r$, $x \in \spt \mu \cap \cball{}{a}{r}$ is
	satisfied for all sufficiently small positive radii $r$.
\end{miniremark}
\begin{lemma} \label{lemma:lower_density_bound}
	Suppose $m, n \in \nat$,\alt{$m < n$,}{} and $\delta > 0$.

	Then there exists a positive number $\varepsilon$ with the following
	property.

	If $a \in \rel^\adim $, $0 < r < \infty$, \alt{$\vdim$,
	$\adim$,}{$\codim$, $\vdim$}, $p$, $U$, and $\mu$ are related as in
	\ref{miniremark:situation1} with $U = \oball{}{a}{r}$, $p=1$, $a \in
	\spt \mu$, and
	\begin{gather*}
		\measureball{\| \delta \mu \|}{\cball{}{a}{\varrho}} \leq
		(2\isoperimetric{\vdim })^{-1} \mu (
		\cball{}{a}{\varrho})^{1-1/\vdim } \quad \text{for $0 <
		\varrho < r$}, \\
		\measureball{\| \delta \mu \|}{\oball{}{a}{r}} \leq
		\varepsilon \, \mu ( \oball{}{a}{r} )^{1-1/\vdim },
	\end{gather*}
	then
	\begin{gather*}
		\measureball{\mu}{\oball{}{a}{r}} \geq ( 1 - \delta )
		\unitmeasure{\vdim } r^\vdim .
	\end{gather*}
\end{lemma}
\begin{proof} [Proof (cf. \protect{\cite[A.10]{mydiss}})]
	If the lemma were false, using \ref{miniremark:good_point}, a
	compactness (see e.g. \cite[Corollary 17.8, Theorem
	42.7]{MR87a:49001}) argument would lead to a contradiction to the
	monotonicity formula (see e.g.  \cite[(17.5)]{MR87a:49001}).
\end{proof}
\begin{remark}
	\ref{miniremark:good_point} and \ref{lemma:lower_density_bound} imply the
	following proposition.

	\emph{If \alt{$\vdim$, $\adim$,}{$\codim$, $\vdim$}, $p$, $U$, $\mu$
	and $\curv$ are as in \ref{miniremark:situation1}, $p = \vdim $, then
	\begin{gather*}
		\density_\ast^\vdim  ( \mu, a ) \geq 1 \quad \text{whenever $a
		\in \spt \mu$ and $\curv ( \{ a \} ) = 0$}.
	\end{gather*}}
	Clearly, the condition $\curv ( \{ a \} ) = 0$ is redundant in case
	\alt{$\delta \mu$ is representable by integration with respect to
	$\mu$}{$\| \delta \mu \|$ is absolutely continuous with respect to
	$\mu$ (i.e. $\delta \mu$ has no singular part with respect to $\mu$)}.
\end{remark}
\begin{definition}
	\alt{For $\adim  \in \nat$ let $\besicovitch{\adim
	}$\index{$\besicovitch{\adim }$} denote the \emph{Besicovitch
	number}\index{Besicovitch number} of dimension $\adim $, i.e. the best
	constant in Besicovitch's covering theorem in $\rel^\adim $.  }{For $k
	\in \nat$ denote by $N (k)$ the best constant in Besicovitch's
	covering theorem in $\rel^k$.}
\end{definition}
\begin{theorem} \label{thm:structure_lower_density}
	Suppose \alt{$\vdim$, $\adim$,}{$\codim$, $\vdim$}, $p$, $U$, $\mu$,
	and $\curv$ are as in \ref{miniremark:situation1}, $p < \vdim $, $0
	\leq s < \infty$, $0 < \varepsilon \leq ( 2 \isoperimetric{\vdim }
	)^{-p/(\vdim -p)}$, $4 \isoperimetric{\vdim } \vdim  < \Gamma <
	\infty$,
	\begin{gather*}
		A = \bigl \{ x \in U \with \density^{\ast \vdim -p} ( \curv,
		x ) < ( \varepsilon / \Gamma )^{\vdim -p} / \unitmeasure{\vdim
		-p} \bigr \},
	\end{gather*}
	denote by $B_i$ for $i \in \nat$ the set of all $x \in U$ such that
	either $\cball{}{x}{1/i} \not \subset U$ or
	\begin{gather*}
		\measureball{\curv}{\cball{}{x}{\varrho}} >
		\varepsilon^{\vdim -p} \, \mu ( \cball{}{x}{\varrho}
		)^{1-p/\vdim } \quad \text{for some $0 < \varrho < 1/i$},
	\end{gather*}
	and denote by $X_i$ for $i \in \nat$ the set of all $a \in U$ such hat
	\begin{gather*}
		\lim_{r \pluslim{0}} \mu \big ( B_i \cap \cball{}{a}{r} \big )
		\big / r^{s\vdim /(\vdim -p)} = 0.
	\end{gather*}

	Then $\classification{B_i}{x}{\cball{}{x}{1/i} \subset U}$ are open
	sets, $X_i$ are Borel sets and
	\begin{gather*}
		\mathcal{H}^s \big ( A \without \union{X_i}{i \in \nat} \big )
		= 0.
	\end{gather*}
\end{theorem}
\begin{proof}
	Clearly, $B_{i+1} \subset B_i$, $X_i \subset X_{i+1}$ and $X_i$ is a
	Borel set for $i \in \nat$. The sets $\{ x
	\in B_i \with \cball{}{x}{1/i} \subset U \}$ are open, as may obtained
	by adapting \cite[2.9.14]{MR41:1976}.

	Define for $i \in \nat$ the set $A_i$ of all $x \in U$ such that
	$\oball{}{x}{1/i} \subset U$ and
	\begin{gather*}
		\measureball{\curv}{\cball{}{x}{\varrho}} \leq ( \varepsilon
		/ \Gamma)^{\vdim -p} \varrho^{\vdim -p} \quad \text{whenever
		$0 < \varrho < 1/i$}.
	\end{gather*}
	The sets $A_i$ are closed (cp. \cite[2.9.14]{MR41:1976}) and satisfy
	$A \subset \union{A_i}{i \in \nat}$. Let $C$ denote the set of all $x
	\in \spt \mu$ such that
	\begin{gather*}
		\limsup_{\varrho \pluslim{0}} \frac{\measureball{\curv}{
		\cball{}{x}{\varrho}}}{\mu ( \cball{}{x}{\varrho} )^{1-p/\vdim
		}} < \varepsilon^{\vdim -p}
	\end{gather*}
	and note $\mu ( U \without C ) = 0$ by \cite[2.9.5]{MR41:1976}. By
	\cite[2.10.6, 2.10.19\,(4)]{MR41:1976} it is enough to prove $a \in
	X_{2i}$ for a point $a \in A_i$ with $\density^s ( \curv \restrict U
	\without A_i, a ) = 0$.

	For this purpose the following assertion will be proved. \emph{For
	each $x \in B_{2i} \cap \oball{}{a}{1/(2i)} \cap C$ there exists $0 <
	\varrho < \infty$ with}
	\begin{gather*}
		\cball{}{x}{\varrho} \subset \oball{}{a}{2|x-a|} \without A_i,
		\quad \measureball{\mu}{\cball{}{x}{\varrho}} <
		\varepsilon^{-\vdim } \curv ( \cball{}{x}{\varrho} )^{\vdim
		/(\vdim -p)}.
	\end{gather*}
	Choose $y \in A_i$ with $|y-x| = \dist (x,A_i)$ and let $J$ be the set
	of all $0 < \varrho < 1/(2i)$ with
	\begin{gather*}
		\mu ( \cball{}{x}{\varrho} ) < \varepsilon^{-\vdim } \curv (
		\cball{}{x}{\varrho} )^{\vdim /(\vdim -p)}.
	\end{gather*}
	Then $J \neq \emptyset$, because $x \in B_{2i}$, $\cball{}{x}{1/(2i)}
	\subset \oball{}{a}{1/i} \subset U$, and, since $x \in C$, $\inf J >
	0$. Therefore $t := \inf J$ satisfies
	\begin{gather*}
		0 < t < 1 / (2i), \quad \measureball{\mu}{\cball{}{x}{t}} \leq
		\varepsilon^{-\vdim } \curv ( \cball{}{x}{t} )^{\vdim /(\vdim
		-p)}, \\
		\measureball{\mu}{\cball{}{x}{\varrho}} \geq
		\varepsilon^{-\vdim } \curv ( \cball{}{x}{\varrho} )^{\vdim
		/(\vdim -p)} \quad \text{for $0 < \varrho < t$}.
	\end{gather*}
	Noting
	\begin{gather*}
		|y-x| = \dist (x,A_i) \leq |x-a| \leq 1/(2i), \quad t + |y-x|
		< 1/i, \\
		\cball{}{x}{t} \subset \cball{}{y}{t+|y-x|} \subset
		\oball{}{y}{1/i} \subset U,
	\end{gather*}
	one estimates
	\begin{multline*}
		\curv ( \cball{}{x}{t} )^{\vdim /(\vdim -p)} \leq \curv (
		\cball{}{y}{t+|y-x|} )^{\vdim /(\vdim -p)} \\
		\leq ( \varepsilon / \Gamma )^\vdim  ( t + |y-x| )^\vdim  <
		\varepsilon^\vdim  2^{-\vdim } ( 1 + |y-x| / t )^\vdim  ( 2
		\vdim \isoperimetric{\vdim } )^{-\vdim } t^\vdim 
	\end{multline*}
	and, using the inequalities derived from the definition of $t$ and
	\ref{miniremark:good_point},
	\begin{gather*}
		\measureball{\mu}{\cball{}{x}{t}} \leq \varepsilon^{-\vdim }
		\curv ( \cball{}{x}{t} )^{\vdim /(\vdim -p)} < 2^{-\vdim } (
		1 + |y-x| / t )^\vdim \measureball{\mu}{\cball{}{x}{t}},
	\end{gather*}
	hence
	\begin{gather*}
		( 1 + |y-x|/t )^\vdim  > 2^\vdim , \quad |y-x| > t
	\end{gather*}
	and the assertion follows by taking $\varrho \in J$ slighly larger
	than $t$.

	Let $0 < r < 1/(2i)$. Then the preceding assertion in conjunction with
	Besicovitch's covering theorem implies the existence of countable,
	\alt{disjointed families}{pairwise disjoint collections} of closed
	balls $F_1, \ldots, F_{\besicovitch{\adim }}$ satisfying
	\begin{gather*}
		B_{2i} \cap \cball{}{a}{r} \cap C \subset
		\alt{{\textstyle\bigcup\bigcup} \{ F_j \with j = 1, \ldots,
		\besicovitch{\adim }
		\}}{{\textstyle\bigcup_{j=1}^{\besicovitch{\adim }}\bigcup_{S
		\in F_j}} S} \subset \oball{}{a}{2r} \without A_i, \\
		\mu (S) < \varepsilon^{-\vdim } \curv ( S )^{\vdim /(\vdim
		-p)} \quad \text{for $S \in \alt{{\textstyle\bigcup} \{ F_j
		\with j = 1, \ldots, \besicovitch{\adim }
		\}}{{\textstyle\bigcup_{j=1}^{\besicovitch{\adim }}}
		F_j}$}.
	\end{gather*}
	Hence
	\begin{align*}
		& \phantom{\leq} \ \mu ( B_{2i} \cap \cball{}{a}{r} ) = \mu (
		B_{2i} \cap \cball{}{a}{r} \cap C ) \\
		& \leq {\textstyle\sum_{j=1}^{\besicovitch{\adim }}}
		{\textstyle\sum_{S \in F_j}} \mu (S) \leq \varepsilon^{-\vdim
		} {\textstyle\sum_{j=1}^{\besicovitch{\adim }}}
		{\textstyle\sum_{S \in F_j}} \curv (S)^{\vdim /(\vdim -p)} \\
		& \leq \varepsilon^{-\vdim }
		{\textstyle\sum_{j=1}^{\besicovitch{\adim }}} \bigl (
		{\textstyle\sum_{S \in F_j}} \curv (S) \bigr)^{\vdim /(\vdim
		-p)} \leq \varepsilon^{-\vdim } \besicovitch{\adim } \curv (
		\oball{}{a}{2r} \without A_i)^{\vdim /(\vdim -p)}
	\end{align*}
	and the conclusion follows by taking
	the limit $r \pluslim{0}$.
\end{proof}
\begin{remark} \label{remark:some_explanations}
	This theorem deserves some explanations.

	First, note that \emph{if $\| \delta \mu \|$ is absolutely continuous
	with respect to $\mu$, then $$\mathcal{H}^{\vdim-p} ( U \without A ) =
	0$$} and \emph{if $p=1$, then $$\mathcal{H}^{\vdim-1} ( X \without A )
	\leq ( \Gamma / \varepsilon )^{\vdim - 1} \unitmeasure{\vdim-1} \|
	\delta \mu \| ( X \without A ) \quad \text{for $X \subset U$}.$$} by
	\cite[2.10.6, 2.10.19\,(3)]{MR41:1976}. These estimates for the size
	of $U \without A$ suggest that the theorem is most useful if $\vdim -
	p \leq s \leq \vdim$.

	Clearly, \emph{if $a \in ( \spt \mu ) \without B_i$, then
	$\cball{}{a}{1/i} \subset U$ and $$( 2 \vdim
	\isoperimetric{\vdim})^{-\vdim} \varrho^\vdim \leq \measureball{\mu}{
	\cball{}{a}{\varrho}} \quad \text{for $0 < \varrho < 1/i$}$$ by
	\ref{miniremark:good_point}}. On the other hand, since the sets
	$\classification{B_i}{x}{\cball{}{x}{1/i} \subset U }$ are open and
	$B_{i+1} \subset B_i$, $X_i \subset X_{i+1}$ for $i \in \nat$, one
	infers that \emph{$\mathcal{H}^s$ almost all $a \in A \cap
	\intersection{B_i}{i \in \nat}$ satisfy $$\lim_{r \pluslim 0}
	\measureball{\mu}{\cball{}{a}{r}} \big / r^{s\vdim/(\vdim-p)} = 0.$$}
\end{remark}
\begin{remark} \label{remark:density_decision}
	Similar to the preceding remark one obtains using
	\ref{lemma:lower_density_bound} instead of \ref{miniremark:good_point}
	that $\mathcal{H}^\vdim $ almost all $x \in U$ satisfy
	\begin{gather*}
		\text{either} \quad \density_\ast^\vdim  ( \mu , x ) \geq 1
		\quad \text{or} \quad \density^{\vdim ^2/(\vdim -p)} ( \mu , x
		) = 0
	\end{gather*}
	and, in case \alt{$\delta \mu$ is representable by integration with
	respect to $\mu$,}{$\| \delta \mu \|$ is absolutely continuous with
	respect to $\mu$,} that $\mathcal{H}^{\vdim -p}$ almost all $x \in U$
	satisfy
	\begin{gather*}
		\text{either} \quad \density_\ast^\vdim  ( \mu , x ) \geq 1
		\quad \text{or} \quad \density^\vdim  ( \mu , x ) = 0.
	\end{gather*}
	Moreover, the exponent $\vdim ^2/(\vdim -p)$ cannot be replaced by any
	larger number as may be seen by taking $\mu \restrict \rel^{\vdim +1}
	\without T$ with $\mu$ as in \ref{example:scaling}. Hence, the same
	holds for the exponent $s\vdim /(\vdim -p)$ in the last equality of
	\ref{remark:some_explanations} if $s=\vdim $.

	Since $\vdim^2/(\vdim-1) \geq \vdim +1$ if $\vdim > 1$, this remark
	seems to be the underlying fact used in the proof of \cite[Theorem
	3.4]{leomas08}.
\end{remark}
\begin{remark}
	It can happen that $\mathcal{H}^\vdim  \bigl ( A \cap ( \spt \mu )
	\cap \intersection{B_i}{i \in \nat} \bigr ) > 0$. In fact taking $\mu
	\restrict \rel^{\vdim +1} \without T$ with $\mu$ as in
	\ref{example:scaling} one sees from \ref{remark:some_explanations} and
	\ref{example:scaling}\,\eqref{item:scaling:estimates} that $T \subset
	\intersection{B_i}{i \in \nat}$.
\end{remark}
\section{A differentiation theorem} \label{app:diff}
In this section the theory of first order differentiation of functions in
Lebesgue spaces defined with respect to a rectifiable varifold, similar to the
one of Calder\'on and Zygmund in \cite{MR0136849} for the special case of
Lebesgue measure, is developed. First, an abstract differentiation theorem for
measures, \ref{thm:big_o_little_o}, is proved which then allows to establish
the differentiation theorem for functions, \ref{thm:diff_mu}. The first part
of that theorem states an approximability result by functions which are
H\"older continuous with exponent $\alpha$ which, in the particular case
$\alpha = 1$ implies a Rademacher type theorem for differentiability in
Lebesgue spaces, see \ref{remark:rademacher}. The second part of
\ref{thm:diff_mu} may in fact be regarded as an application of this theory and
is designed for use in \cite{snulmenn:poincare}.
\begin{theorem} \label{thm:big_o_little_o}
	Suppose \alt{$\vdim$, $\adim$,}{$\codim$, $\vdim$}, $p$, $U$, and
	$\mu$ are as in \ref{miniremark:situation1}, $\nu$ measures $U$ with
	$\nu ( U \without \spt \mu ) = 0$, $A$ is $\mu$ measurable with $\nu
	(A) = 0$, and $1 \leq q < \infty$. In case $p < \vdim $ additionally
	suppose for some $1 \leq r \leq \infty$ and some nonnegative function
	$f \in \Lp{r}_\mathrm{loc} ( \mu)$ that
	\begin{gather*}
		\nu = f \mu \quad \text{and} \quad q \leq 1 + ( 1 - 1/r )
		\frac{p}{\vdim -p}.
	\end{gather*}

	Then for $\mathcal{H}^\vdim $ almost all $a \in A$
	\begin{gather*}
		\limsup_{s \pluslim{0}} \nu ( \cball{}{a}{s} ) \big / s^{\vdim
		q} \quad \text{equals either $0$ or $\infty$}.
	\end{gather*}
\end{theorem}
\begin{proof}
	For $i \in \nat$ let $B_i$ denote the set of all $x \in U$ such that
	either $\cball{}{x}{1/i} \not \subset U$ or
	\begin{gather*}
		\measureball{\| \delta \mu \|}{\cball{}{x}{\varrho}} > ( 2
		\isoperimetric{\vdim } )^{-1} \mu ( \cball{}{x}{\varrho}
		)^{1-1/\vdim } \quad \text{for some $0 < \varrho < 1/i$}.
	\end{gather*}

	First, the case $A \subset \{ x \in U \with \density^{\ast \vdim } (
	\mu, x) > 0 \}$ will be treated. In this case $A$ is \alt{countably
	$\mathcal{H}^\vdim $ measurable\footnote{Recall that a set is called
	\emph{countably $\phi$ measurable} if and only if it is expressible as
	the union of a countable collection of $\phi$ measurable sets, each
	with finite $\phi$ measure.}}{measurable and $\sigma$ finite with
	respect to $\mathcal{H}^\vdim $} by
	\cite[2.10.19\,(1)\,(3)]{MR41:1976}. Hence one may assume $A$ to be
	compact. Define
	\begin{gather*}
		A_i = \{ a \in A \with
		\text{$\measureball{\nu}{\cball{}{a}{s}}/ \leq i\, s^{\vdim
		q}$ for $0 < s < 1/i$} \}
	\end{gather*}
	whenever $i \in \nat$, $1/i < \dist ( A, \rel^\adim  \without U)$. The
	sets $A_i$ are compact (cp.  \cite[2.9.14]{MR41:1976}) and their union
	equals
	\begin{gather*}
		\big \{ a \in A \with \limsup_{s \pluslim{0}} \nu (
		\cball{}{a}{s} ) / s^{\vdim q} < \infty \big \}.
	\end{gather*}
	It therefore suffices to show for each $i \in \nat$ with $1/i < \dist
	( A, \rel^\adim  \without U )$
	\begin{gather*}
		\lim_{s \pluslim{0}} \nu ( \cball{}{a}{s} ) \big / s^{\vdim q}
		= 0 \quad \text{for $\mathcal{H}^\vdim $ almost all $a \in
		A_i$}.
	\end{gather*}
	In fact, this equality will be proved for all $a \in A_i$ satisfying
	\begin{gather*}
		\| \delta \mu \| ( \{ a \} ) = 0, \quad \density^\vdim  ( \mu
		\restrict U \without A_i , a ) = 0, \quad \density^\vdim  (
		f^r \mu, a ) = 0 \quad \text{if $r < \infty$}, \\
		\limsup_{s \pluslim{0}} \mu ( B_j \cap \cball{}{a}{s} ) \big /
		s^{\vdim ^2/(\vdim -p)} = 0 \quad \text{for some $j \in \nat$,
		$j \geq 2i$, if $p < \vdim $}
	\end{gather*}
	as $\mathcal{H}^\vdim $ almost all $a \in A_i$ do according to
	\cite[2.10.19\,(3)\,(4)]{MR41:1976} and
	\ref{thm:structure_lower_density}.

	In case $p = \vdim $ one chooses $j \in \nat$, $j \geq 2i$, using
	\ref{miniremark:good_point} such that
	\begin{gather*}
		B_j \cap \cball{}{a}{1/j} = \emptyset.
	\end{gather*}
	Let $0 < s < 1/j$. For $x \in \cball{}{a}{s} \cap ( \spt \mu )
	\without ( B_j \cup A_i )$ there exists $y \in A_i$ with $|x-y| =
	\dist (x,A_i)$, hence
	\begin{gather*}
		t := |x-y| \leq |x-a| \leq s < 1/j \leq 1/(2i), \\
		\cball{}{x}{|x-y|/2} \subset \cball{}{y}{3|x-y|/2} \cap
		\cball{}{a}{2s} \without A_i, \\
		\measureball{\nu}{\cball{}{x}{t/2}} \leq \measureball{\nu}{
		\cball{}{y}{3t/2}} \leq i 3^{\vdim q} ( t/2 )^{\vdim q} \leq c
		\, \mu ( \cball{}{x}{t/2})^q
	\end{gather*}
	where $c = i 3^{\vdim q} ( 2 \isoperimetric{\vdim } \vdim  )^{\vdim
	q}$. Therefore one infers from Besicovitch's covering theorem the
	existence of countable, \alt{disjointed families}{pairwise disjoint
	collections} $F_1, \ldots, F_{\besicovitch{\adim }}$ of closed balls
	such that
	\begin{gather*}
		\cball{}{a}{s} \cap ( \spt \mu ) \without ( B_j \cup A_i )
		\subset \alt{{\textstyle\bigcup \bigcup} \{ F_k \with k = 1,
		\ldots, \besicovitch{\adim }
		\}}{{\textstyle\bigcup_{k=1}^{\besicovitch{\adim }}\bigcup_{S
		\in F_k} S}} \subset \cball{}{a}{2s} \without A_i, \\
		\nu (S) \leq c \, \mu (S)^q \quad \text{whenever $S \in \alt{
		{\textstyle\bigcup} \{ F_k \with k = 1, \ldots,
		\besicovitch{\adim }
		\}}{{\textstyle\bigcup_{k=1}^{\besicovitch{\adim }}} F_k}$},
	\end{gather*}
	hence
	\begin{gather*}
		\nu ( \cball{}{a}{s} \without B_j ) = \nu ( \cball{}{a}{s}
		\cap ( \spt \mu ) \without ( B_j \cup A_i ) ) \leq c
		\besicovitch{\adim } \, \mu ( \cball{}{a}{2s} \without A_i)^q,
		\\
		\lim_{s \pluslim{0}} \nu ( \cball{}{a}{s} \without B_j ) \big
		/ s^{\vdim q} = 0.
	\end{gather*}

	To conclude the proof of the first case, one observes
	\begin{gather*}
		\nu ( B_j \cap \cball{}{a}{s} ) = 0 \quad \text{if $p=\vdim
		$}, \\
		\nu ( B_j \cap \cball{}{a}{s} ) \leq \mu ( B_j \cap
		\cball{}{a}{s} )^{1-1/r} \| f \|_{\Lp{r} ( \mu \restrict
		\cball{}{a}{s})} \quad \text{if $p < \vdim $}
	\end{gather*}
	implying
	\begin{gather*}
		\lim_{s \pluslim{0}} \nu ( B_j \cap \cball{}{a}{s} ) \big /
		s^{\vdim q} = 0
	\end{gather*}
	because $( 1-1/r) \frac{\vdim }{\vdim -p} + 1/r \geq q$ in case $p <
	\vdim $.

	It remains to treat the case $A \subset \{ x \in U \with
	\density^\vdim  ( \mu, x) = 0 \}$. Using \ref{miniremark:good_point} and
	\ref{remark:density_decision} one obtains
	\begin{gather*}
		\text{$A \cap \spt \mu$ is countable} \quad \text{if $p=\vdim
		$}, \\
		\density^{\vdim ^2/(\vdim -p)} ( \mu, a ) = 0 \quad \text{for
		$\mathcal{H}^\vdim $ almost all $a \in A$} \quad \text{if $p <
		\vdim $}
	\end{gather*}
	and the claim follows by using H{\"o}lder's inequality as in the
	preceding paragraph noting by \cite[2.10.19\,(4)]{MR41:1976}
	\begin{gather*}
		\density^\vdim  ( f^r \mu , a ) = 0 \quad \text{for
		$\mathcal{H}^\vdim $ almost all $a \in A$} \quad \text{if $r<
		\infty$}. \qedhere
	\end{gather*}
\end{proof}
\begin{remark}
	This theorem generalises \cite[2.9.17]{MR41:1976} and \cite[Theorem
	10\,(ii)]{MR0136849}. The case treated by Federer roughly corresponds
	to the case $p=\vdim $, $q=1$ with $\mu$ satisfying a doubling
	condition. The case treated by Calder\'on and Zygmund corresponds to
	\alt{$p=\vdim =\adim $}{$p=\vdim$, $\codim=0$}, $\mu =
	\mathcal{L}^\adim $ and $\nu$ \alt{representable by
	integration}{absolutely continuous} with respect to $\mu$. The method
	of proof is based on Federer's proof and
	\ref{thm:structure_lower_density} is used because of the absence of a
	doubling condition.
\end{remark}
\begin{remark} \label{remark:precise1}
	If $q = 1$, the condition $\nu ( U \without \spt \mu ) = 0$ cannot be
	omitted as may be seen from \cite[2.9.18\,(2)]{MR41:1976}.
\end{remark}
\begin{remark} \label{remark:precise2}
	If $p < \vdim $ the condition $q \leq 1 + ( 1-1/r)p/(\vdim -p)$ cannot
	be omitted as can be shown using \ref{example:scaling}. In fact given
	$\mu$ and $T$ as in \ref{example:scaling} a counterexample is provided
	by $\nu := \mu \restrict \rel^{\vdim +1} \without T$ in case $r =
	\infty$ and if $1 < r < \infty$ applying
	\ref{example:scaling}\,\eqref{example:scaling:last} with $s = nq$ and
	$\hoelder_1 q_1 = \hoelder_2 q_2$ slightly larger than $\frac{\vdim
	p}{\vdim - p}$ yields a function $f$ such that $\nu := f \mu$ does not
	satisfy the conclusion of \ref{thm:big_o_little_o}. Finally, if $r=1$
	the condition is also violated for a slightly larger $r$, hence
	reducing this case to the previous one.
\end{remark}
\begin{remark}
	Note that the preceding two remarks remain valid if $\mathcal{H}^\vdim
	$ is replaced by $\mu$ in the conclusion of \ref{thm:big_o_little_o}.
\end{remark}
\begin{definition} \label{def:fint}
	Whenever $A$ is $\phi$ measurable set with $0 < \phi (A) < \infty$ and
	$f \in \Lp{1} ( \phi \restrict A )$ one defines $\fint_A f \ud \phi =
	\phi(A)^{-1} \int_A f \ud \phi$.
\end{definition}
\begin{theorem} \label{thm:diff_mu}
	Suppose $\vdim$, $\codim$, $p$, $U$, and $\mu$ are as in
	\ref{miniremark:situation1}, $Z$ is a separable Banach space, $f : U
	\to Z$ is $\mu$ measurable, $0 < \alpha \leq 1$, $1 \leq q < \infty$,
	and $A$ is the set of all $x \in \spt \mu$ such that
	\begin{gather*}
		\limsup_{\varrho \pluslim{0}} \varrho^{-\alpha q}
		{\textstyle\fint_{\cball{}{x}{\varrho}}} | f ( \xi ) - z |^q
		\ud \mu (\xi) < \infty \quad \text{for some $z \in Z$}.
	\end{gather*}
	In case $p < \vdim$ additionally suppose that $f \in \Lploc{r} ( \mu,
	Z )$ for some $1 \leq r \leq \infty$ satisfying
	\begin{gather*}
		\alpha q / \vdim \leq \Big ( 1 - \frac{q}{r} \Big )
		\frac{p}{\vdim-p}.
	\end{gather*}
	
	Then $A$ is a Borel set and the following two statements hold:
	\begin{enumerate}
		\item \label{item:diff_mu:approx} For every $\varepsilon > 0$
		there exists a function $g : U \to Z$ which locally satisfies
		a H{\"o}lder condition with exponent $\alpha$ such that
		\begin{gather*}
			\mu ( A \cap \{ x \with f(x) \neq g(x) \} ) \leq
			\varepsilon.
		\end{gather*}
		Moreover, for every function $g$ which locally satisfies a
		H{\"o}lder condition with exponent $\alpha$ there holds
		\begin{gather*}
			\lim_{\varrho \pluslim{0}} \varrho^{-\alpha q}
			{\textstyle\fint_{\cball{}{x}{\varrho}}} | f (\xi) -
			g(\xi) |^q \ud \mu ( \xi ) = 0
		\end{gather*}
		for $\mu$ almost all $x \in A$ with $f(x) = g(x)$.
		\item \label{item:diff_mu:scaling_bad_set} If $\varepsilon >
		0$, $D_i(a)$ denotes for $a \in \dmn f$, $i \in \nat$
		the set of all $x \in U$ such that either $\cball{}{x}{1/i}
		\not \subset U$ or
		\begin{gather*}
			{\textstyle\int_{\cball{}{x}{\varrho}}} | f ( \xi ) -
			f (a) |^q \ud \mu ( \xi ) > \varepsilon \,
			\measureball{\mu}{\cball{}{x}{\varrho}} \quad
			\text{for some $0 < \varrho < 1/i$},
		\end{gather*}
		$Y_i$ denotes for $i \in \nat$ the set of all $a \in U$
		such that 
		\begin{gather*}
			\lim_{r \pluslim{0}} \mu ( D_i(a) \cap \cball{}{a}{r} )
			/ r^{\vdim+\alpha q} = 0,
		\end{gather*}
		then the sets $Y_i$ are $\mu$ measurable and
		\begin{gather*}
			\mu \big ( A \without {\textstyle\bigcup} \{ Y_i \with
			i \in \nat \} \big ) = 0.
		\end{gather*}
	\end{enumerate}
\end{theorem}
\begin{proof} [Proof of \eqref{item:diff_mu:approx}]
	Let $\pi : \rel^\adim \times Z \to \rel^\adim$ denote the
	projection and for $i \in \nat$ let $E_i$ denote the set of all
	$(x,z) \in \spt \mu \times Z$ such that $\oball{}{x}{1/i} \subset U$
	and
	\begin{gather*}
		{\textstyle\fint_{\cball{}{x}{\varrho}}} | f(\xi) - z|^q \ud
		\mu (\xi) \leq i \varrho^{\alpha q} \quad \text{whenever $0 <
		\varrho < 1/i$}.
	\end{gather*}
	Then $E_i$ is closed (cp. \cite[2.9.14]{MR41:1976}), $\pi | E_i$ is
	univalent, and both $\pi \lIm E_i \rIm$ and $A = \bigcup \{ \pi \lIm
	E_i \rIm \with i \in \nat \}$ are Borel sets by
	\cite[2.2.10]{MR41:1976}.

	To prove the first part of \eqref{item:diff_mu:approx}, the problem is
	reduced to the case $\mu = \mathcal{L}^\vdim \restrict K$ for some
	compact set $K$ (not necessarily satisfying a condition on $\delta
	\mu$) via \cite[3.2.18]{MR41:1976}. This case can then be treated by
	adapting \cite[3.1.8, 3.1.14]{MR41:1976}, see also
	\cite[\printRoman{6}.2.2.2]{MR0290095}.

	Concerning the second half of \eqref{item:diff_mu:approx}, one
	observes that every such function $g$ satisfies
	\begin{gather*}
		\limsup_{\varrho \pluslim{0}} \varrho^{-\alpha q}
		{\textstyle\fint_{\cball{}{x}{\varrho}}} | f ( \xi ) - g( \xi
		)|^q \ud \mu ( \xi ) < \infty
	\end{gather*}
	for $\mathcal{L}^\vdim$ almost all $x \in A$ with $f(x) = g(x)$ and
	\ref{thm:big_o_little_o} may be applied with $\nu$, $r$, $q$, $A$
	replaced by $|f-g|^q \mu$, $r/q$, $1 + \alpha q / \vdim$,
	$\classification{A}{x}{f(x) = g(x)}$ if $p < \vdim$ and $|f-g|^q \mu$,
	$\infty$, $1 + \alpha q / \vdim$, $\classification{A}{x}{f(x) = g(x)}$
	else.
\end{proof}
\begin{proof} [Proof of \eqref{item:diff_mu:scaling_bad_set}]
	For any $0 < \varrho < \infty$, $x \in \rel^\adim$ denote by
	$b_{x,\varrho}$ the characteristic function of $\cball{}{x}{\varrho}$,
	define $U_i = \classification{U}{x}{\dist (x, \rel^\adim \without U )
	> 1/i}$ and observe that the function mapping $(a,x,\xi) \in ( \dmn f
	) \times U \times ( \dmn f )$ onto
	\begin{gather*}
		b_{x,\varrho} (\xi) | f ( \xi ) - f (a) |^q - \varepsilon
		b_{x,\varrho} ( \xi )
	\end{gather*}
	is $\mu \times \mu \times \mu$ measurable for every $0 < \varrho <
	\infty$. Applying Fubini's theorem, one infers that the function
	mapping $(a,x) \in ( \dmn f ) \times U_i$ onto
	\begin{gather*}
		\sup \big \{ {\textstyle\int_{\cball{}{x}{\varrho}}} | f (
		\xi ) - f (a) |^q \ud \mu ( \xi ) - \varepsilon \,
		\measureball{\mu}{\cball{}{x}{\varrho}} \with 0 < \varrho <
		1/i \big \}
	\end{gather*}
	is $\mu \times ( \mu \restrict U_i )$ measurable for each $i \in
	\nat$, since the supremum by restricted to a countable, dense
	subset of $\{ \varrho \with 0 < \varrho < 1/i \}$. For the same reason
	\begin{gather*}
		\sup \big \{ r^{-m-\alpha q} \mu ( X_i(a) \cap \cball{}{a}{r}
		) \with 0 < r < 1/j \big \}
	\end{gather*}
	depends $\mu$ measurably on $a$ for each $i,j \in \nat$.
	Therefore the sets $Y_i$ are $\mu$ measurable.

	For $i \in \nat$ let $A_i$ denote the set of all $a \in ( \dmn f
	) \cap ( \spt \mu )$ such that $\oball{}{a}{1/i} \subset U$ and
	whenever $0 < \varrho < 1/i$
	\begin{gather*}
		\measureball{\mu}{\cball{}{a}{\varrho}} \leq i \varrho^\vdim,
		\quad {\textstyle\fint_{\cball{}{a}{\varrho}}} | f ( \xi ) - f
		(a) |^q \ud \mu ( \xi ) \leq i \varrho^{\alpha q}.
	\end{gather*}
	$A_i$ are $\mu$ measurable sets as may be verified using the first
	paragraph of the proof of \eqref{item:diff_mu:approx} and noting the
	fact that the last condition may be replaced by the two conditions
	\begin{gather*}
		a \in \pi \lIm E_i \rIm , \quad \lim_{\varrho \pluslim{0}}
		{\textstyle\fint_{\cball{}{a}{\varrho}}} f (\xi) \ud \mu (
		\xi ) = f(a).
	\end{gather*}
	Note $\mu ( A \without \bigcup \{ A_i \with i \in \nat \} ) =
	0$. For $i \in \nat$ let $C_i$ denote the set of all $x \in \spt
	\mu$ such that either $\cball{}{x}{1/i} \not \subset U$ or
	\begin{gather*}
		\| \delta \mu \| \, \cball{}{x}{\varrho} > ( 2
		\isoperimetric{\vdim} )^{-1} \mu ( \cball{}{x}{\varrho}
		)^{1-1/\vdim} \quad \text{for some $0 < \varrho < 1/i$}.
	\end{gather*}
	Moreover, define
	\begin{gather*}
		X_i = \classification{U}{x}{\density^{\vdim+\alpha q} ( \mu
		\restrict C_i, x ) = 0} \quad \text{for $i \in \nat$},
	\end{gather*}
	note $\vdim+\alpha q \leq \vdim^2/(\vdim-p)$ if $p < \vdim$, and
	observe by \ref{thm:structure_lower_density} in case $p < \vdim$, by
	\ref{miniremark:good_point} in case $p=\vdim$, that
	\begin{gather*}
		\mu ( U \without {\textstyle\bigcup} \{ X_i \with i \in \nat
		\} ) = 0.
	\end{gather*}
	Using \eqref{item:diff_mu:approx}, one constructs sequences $K_i$ of
	compact subsets of $U$ and $g_i : U \to Z$ such that
	\begin{gather*}
		\text{$K_i \subset A_j$ for some $j \in \nat$}, \quad
		f|K_i = g_i | K_i, \\
		\text{$g_i$ locally satisfies a H{\"o}lder condition with
		exponent $\alpha$}, \\
		\mu ( A \without {\textstyle\bigcup} \{ K_i \with i \in
		\nat \} ) = 0.
	\end{gather*}
	Also note $A_i \subset A_{i+1}$, $C_i \supset C_{i+1}$, and $X_i
	\subset X_{i+1}$ for $i \in \nat$.

	From the observations of the preceding paragraph, \cite[2.10.6,
	2.10.19\,(4)]{MR41:1976} and \eqref{item:diff_mu:approx} it follows
	that it is enough to prove $a \in \bigcup \{ X_j \with j \in
	\nat \}$ whenever $a \in A$ satisfies for some $i \in
	\nat$, some compact set $K$, and some $g : U \to Z$
	\begin{gather*}
		a \in X_i, \quad a \in K \subset A_i , \quad
		\density^\vdim ( \mu \restrict U \without K, a ) = 0,
		\quad g|K = f|K, \\
		\text{$g$ locally satisfies a H{\"o}lder condition with
		exponent $\alpha$}, \\
		r^{-\vdim-\alpha q} {\textstyle\int_{\cball{}{a}{r}}} | f (
		\xi) - g (\xi) |^q \ud \mu ( \xi ) \to 0 \quad \text{as $r
		\pluslim{0}$}.
	\end{gather*}
	For this purpose define $h = \sup \{ | g(x)-g(y)|/|x-y|^\alpha \with
	x,y \in K, x \neq y \}$, choose $j \in \nat$, $j \geq 2i$, and $0 < R <
	1/(2i)$ satisfying
	\begin{gather*}
		2^{q-1} i^2 \big ( ( 1/j + R )^{\alpha q} + h^q (2R)^{\alpha
		q} \big ) \leq \varepsilon 2^{-\vdim} ( 2
		\isoperimetric{\vdim} \vdim)^{-\vdim}.
	\end{gather*}

	Next, it will be shown\emph{
	\begin{gather*}
		{\textstyle\int_{\cball{}{x}{\varrho}}} | f ( \xi ) - f ( a
		) |^q \ud \mu ( \xi ) \leq \varepsilon 2^{-\vdim} ( 1 +
		|\zeta-x|/\varrho )^\vdim
		\measureball{\mu}{\cball{}{x}{\varrho}}
	\end{gather*}
	whenever $x \in \spt \mu \cap \cball{}{a}{r} \without C_i$, $\zeta
	\in K$, $|\zeta-x| = \dist (x,K)$, $0 < r \leq R$, $0 < \varrho <
	1/j$}. Noting
	\begin{gather*}
		\varrho + | \zeta - x | < 1/j + |x-a| \leq 1/j + R < 1/i, \\
		\cball{}{x}{\varrho} \subset \cball{}{\zeta}{\varrho+|\zeta-x|}
		\subset \oball{}{\zeta}{1/i} \subset U, \\
		| \zeta - a | \leq | \zeta - x | + | x - a | \leq 2 | x-a |
		\leq 2R, \\
		2^{q-1} i^2 \big ( ( \varrho + |\zeta-x| )^{\alpha q} + h^q |
		\zeta - a |^{\alpha q} \big ) \leq \varepsilon 2^{-\vdim} ( 2
		\isoperimetric{\vdim} \vdim )^{-\vdim},
	\end{gather*}
	one estimates
	\begin{multline*}
		{\textstyle\int_{\cball{}{x}{\varrho}}} | f ( \xi ) - f (a)
		|^q \ud \mu ( \xi ) \leq {\textstyle\int_{\cball{}{\zeta}{
		\varrho + | \zeta - x | }}} | f ( \xi ) - f (a) |^q \ud \mu (
		\xi ) \\
		\begin{aligned}
			& \leq 2^{q-1} \big ( {\textstyle\int_{\cball{}{
			\zeta}{\varrho + | \zeta - x | }}} | f ( \xi ) - f (
			\zeta ) |^q \ud \mu ( \xi ) + | f ( \zeta ) - f (a)
			|^q \measureball{\mu}{\cball{}{\zeta}{\varrho + |
			\zeta - x |}} \big ) \\
			& \leq 2^{q-1} i \big ( ( \varrho + | \zeta - x |
			)^{\alpha q} + h^q | \zeta - a |^{\alpha q} \big )
			\measureball{\mu}{\cball{}{\zeta}{\varrho + | \zeta -
			x | }} \\
			& \leq \varepsilon 2^{-\vdim} ( 2 \isoperimetric{\vdim}
			\vdim )^{-\vdim} ( 1 + | \zeta - x | / \varrho )^\vdim
			\varrho^\vdim
		\end{aligned}
	\end{multline*}
	and \ref{miniremark:good_point} implies the assertion. Therefore, if
	\begin{gather*}
		{\textstyle\int_{\cball{}{x}{\varrho}}} | f ( \xi ) - f (a)
		|^q \ud \mu ( \xi ) > \varepsilon \,
		\measureball{\mu}{\cball{}{x}{\varrho}},
	\end{gather*}
	then
	\begin{gather*}
		( 1 + | \zeta - x | / \varrho )^\vdim > 2^\vdim, \quad \varrho
		< | \zeta - x | \leq | x-a | \leq r, \quad | x-a | + \varrho <
		2r, \\
		\cball{}{x}{\varrho} \subset \oball{}{a}{2r} \without K
		\subset U.
	\end{gather*}

	This implies that for each $x \in \spt \mu \cap \cball{}{a}{r} \cap
	D_j (a) \without C_i$ with $0 < r \leq R$ there exists $0 < \varrho <
	1/j$ such that
	\begin{gather*}
		\cball{}{x}{\varrho} \subset \oball{}{a}{2r} \without K
		\subset U, \quad {\textstyle\int_{\cball{}{x}{\varrho}}} | f
		( \xi ) - f (a) |^q \ud \mu ( \xi ) > \varepsilon \,
		\measureball{\mu}{\cball{}{x}{\varrho}},
	\end{gather*}
	because $a \in A_i$, $x \in \cball{}{a}{r}$ implies $\cball{}{x}{1/j}
	\subset U$. Hence one infers from Besicovitch's covering theorem
	\begin{gather*}
		\mu ( \cball{}{a}{r} \cap D_j (a) \without C_i ) \leq
		\besicovitch{\adim} \varepsilon^{-1}
		{\textstyle\int_{\oball{}{a}{2r} \without K}} | f ( \xi ) -
		f (a) |^q \ud \mu ( \xi )
	\end{gather*}
	for $0 < r \leq R$. Recalling $a \in X_i$, the proof may be concluded
	by showing
	\begin{gather*}
		r^{-\vdim-\alpha q} {\textstyle\int_{\oball{}{a}{2r} \without
		K}} | f ( \xi ) - f (a) |^q \ud \mu ( \xi ) \to 0 \quad
		\text{as $r \pluslim{0}$}
	\end{gather*}
	which is a consequence of
	\begin{gather*}
		{\textstyle\int_{\oball{}{a}{2r} \without K}} | f ( \xi ) -
		f (a) |^q \ud \mu ( \xi ) \\
		\leq 2^{q-1} \big ( {\textstyle\int_{\oball{}{a}{2r}}} | f (
		\xi ) - g ( \xi ) |^q \ud \mu ( \xi ) +
		{\textstyle\int_{\oball{}{a}{2r} \without K}} | g ( \xi ) -
		g ( a ) |^q \ud \mu ( \xi ), \\
		{\textstyle\int_{\oball{}{a}{2r} \without K}} | g ( \xi ) - g
		(a) |^q \ud \mu ( \xi ) \leq \mu ( \oball{}{a}{2r} \without
		K ) (h_0)^q (2r)^{\alpha q}
	\end{gather*}
	for $0 < r \leq R$ with $h_0 = \sup \{ |g(x)-g(y)|/|x-y|^\alpha \with
	x,y \in \cball{}{a}{R}, x \neq y \}$.
\end{proof}
\begin{remark}
	If $p < \vdim$ the assumption $\alpha q / \vdim \leq
	(1-q/r)p/(\vdim-p)$ cannot be omitted in order to obtain the second
	part of \eqref{item:diff_mu:approx} as may be seen from the family of
	examples constructed in \ref{example:scaling}; in fact one can take
	$\alpha_1 = \alpha_2 = \alpha$, $q_1 = q_2 = q$, and $f =
	\chi_{\rel^{\vdim+1} \without T}$ in case $r=\infty$, and in case $r <
	\infty$ one can assume $q < r$ and apply
	\ref{example:scaling}\,\ref{example:scaling:last} with $r$, $s$,
	$\alpha_1 = \alpha_2$, $q_1 = q_2$ replaced by $r/q$, $\alpha q +
	\vdim$, $1$ and a number slighly larger than $\vdim p/(\vdim-p)$ to
	obtain a function $f \in \Lploc{r/q} ( \mu )$ such that the second
	statement of \eqref{item:diff_mu:approx} does not hold for $f$, $g$
	replaced by $f^{1/q}$, $0$.
\end{remark}
\begin{remark} \label{remark:rademacher}
	If $\dim Z < \infty$ and $\alpha = 1$, \eqref{item:diff_mu:approx} in
	conjunction with \cite[3.2.18, 3.1.16]{MR41:1976} implies that for
	$\mu$ almost all $a \in A$
	\begin{gather*}
		\lim_{\varrho \pluslim{0}}
		{\textstyle\fint_{\cball{}{a}{\varrho}}} ( | f ( \xi ) - f (a)
		- \left < ( T_a \mu ) ( \xi-a), ( \mu, \vdim ) \ap Df (a)
		  \right > |/|\xi-a|)^q \ud \mu ( \xi ) = 0
	\end{gather*}
	where the notion of approximate differentials, see
	\cite[3.2.16]{MR41:1976}, is employed.
\end{remark}
\begin{remark} \label{remark:final}
	\eqref{item:diff_mu:scaling_bad_set} can be seen in two ways as a
	refinement of the simple fact that
	\begin{gather*}
		\density^{\ast \vdim+\alpha q} ( \mu \restrict
		\classification{U}{x}{|f(x)-f(a)|^q>\varepsilon} , a )
		\leq \varepsilon^{-1} \density^{\ast \vdim+\alpha q} ( | f (
		\cdot ) - f(a) |^q \mu , a ) < \infty
	\end{gather*}
	whenever $a \in A$. Firstly, $|f(x)-f(a)|^q > \varepsilon$ is replaced
	in the definition of $D_i(a)$ by $\fint_{\cball{}{x}{\varrho}} | f
	(\xi)-f(a) |^q \ud \mu ( \xi ) > \varepsilon$ for some $0 < \varrho <
	1/i$.  Secondly, in the conclusion $\density^{\vdim+\alpha q} ( \mu
	\restrict D_i (a), a) = 0$ occurs instead of $\density^{\ast
	\vdim+\alpha q} ( \mu \restrict D_i (a), a ) < \infty$. Whereas the
	first improvement is vital for the applications in \cite{snulmenn:c2},
	the second one is only used under the stronger assumption
	\begin{gather*}
		\lim_{r \pluslim{0}} r^{-\vdim - \alpha q}
		{\textstyle\int_{\cball{}{x}{r}}} | f ( \xi ) - z |^q \ud
		\mu ( \xi ) = 0 \quad \text{for some $z \in Z$}
	\end{gather*}
	for $\mu$ almost all $x \in U$.
\end{remark}
\bibliography{ohne_dublikate} \bibliographystyle{alpha}

\newcommand{\noopsort}[1]{} \newcommand{\singleletter}[1]{#1} \def\cprime{$'$}
  \def\cprime{$'$}
\begin{thebibliography}{Men08b}

\bibitem[All72]{MR0307015}
William~K. Allard.
\newblock On the first variation of a varifold.
\newblock {\em Ann. of Math. (2)}, 95:417--491, 1972.

\bibitem[Alm00]{MR1777737}
Frederick~J. Almgren, Jr.
\newblock {\em Almgren's big regularity paper}, volume~1 of {\em World
  Scientific Monograph Series in Mathematics}.
\newblock World Scientific Publishing Co. Inc., River Edge, NJ, 2000.
\newblock $Q$-valued functions minimizing Dirichlet's integral and the
  regularity of area-minimizing rectifiable currents up to codimension 2, With
  a preface by Jean E.\ Taylor and Vladimir Scheffer.

\bibitem[AS94]{MR1285779}
Gabriele Anzellotti and Raul Serapioni.
\newblock {$\mathcal{C}^k$}-rectifiable sets.
\newblock {\em J. Reine Angew. Math.}, 453:1--20, 1994.

\bibitem[Bra78]{MR485012}
Kenneth~A. Brakke.
\newblock {\em The motion of a surface by its mean curvature}, volume~20 of
  {\em Mathematical Notes}.
\newblock Princeton University Press, Princeton, N.J., 1978.

\bibitem[CZ61]{MR0136849}
A.-P. Calder{\'o}n and A.~Zygmund.
\newblock Local properties of solutions of elliptic partial differential
  equations.
\newblock {\em Studia Math.}, 20:171--225, 1961.

\bibitem[Fed69]{MR41:1976}
Herbert Federer.
\newblock {\em Geometric measure theory}.
\newblock Die Grundlehren der ma\-the\-ma\-ti\-schen Wissenschaften, Band 153.
  Springer-Verlag New York Inc., New York, 1969.

\bibitem[LM08]{leomas08}
G.~P. Leonardi and S.~Masnou.
\newblock Locality of the mean curvature of rectifiable varifolds.
\newblock URL:~http://cvgmt.sns.it/papers/leomas08/~, 2008.

\bibitem[Men08a]{mydiss}
Ulrich Menne.
\newblock {\em $\mathcal{C}^2$ rectifiability and {$Q$} valued functions}.
\newblock PhD thesis, Universit{\"a}t T{\"u}bingen, \noopsort{a}2008.
\newblock
  URL:~http://{\linebreak[0]}tobias-{\linebreak[0]}lib.{\protect{\linebreak[0]%
}}ub.{\protect{\linebreak[0]}}uni-tuebingen.{\linebreak[0]}de/{\linebreak[0]}v%
olltexte/{\linebreak[0]}2008/{\linebreak[0]}3518~.

\bibitem[Men08b]{snulmenn:poincare}
Ulrich Menne.
\newblock {A} {S}obolev {P}oincar{\'e} type inequality for integral varifolds,
  \noopsort{c}2008.
\newblock in preparation.

\bibitem[Men08c]{snulmenn:c2}
Ulrich Menne.
\newblock {S}econd order rectifiability of integral varifolds of locally
  bounded first variation, \noopsort{d}2008.
\newblock in preparation.

\bibitem[Sch01]{MR1906780}
Reiner Sch{\"a}tzle.
\newblock Hypersurfaces with mean curvature given by an ambient {S}obolev
  function.
\newblock {\em J. Differential Geom.}, 58(3):371--420, 2001.

\bibitem[Sch04a]{MR2064971}
Reiner Sch{\"a}tzle.
\newblock Quadratic tilt-excess decay and strong maximum principle for
  varifolds.
\newblock {\em Ann. Sc. Norm. Super. Pisa Cl. Sci. (5)}, 3(1):171--231,
  \noopsort{a}2004.

\bibitem[Sch04b]{rsch:willmore}
Reiner Sch{\"a}tzle.
\newblock {L}ower {S}emicontinuity of the {W}illmore {F}unctional for
  {C}urrents.
\newblock Preprint No. 167, Sonderforschungsbereich 611, Bonn,
  \noopsort{b}2004.

\bibitem[Sim83]{MR87a:49001}
Leon Simon.
\newblock {\em Lectures on geometric measure theory}, volume~3 of {\em
  Proceedings of the Centre for Mathematical Analysis, Australian National
  University}.
\newblock Australian National University Centre for Mathematical Analysis,
  Canberra, 1983.

\bibitem[Ste70]{MR0290095}
Elias~M. Stein.
\newblock {\em Singular integrals and differentiability properties of
  functions}.
\newblock Princeton Mathematical Series, No. 30. Princeton University Press,
  Princeton, N.J., 1970.

\end{thebibliography}
\end{document}